\documentclass[twocolumn]{autart}     

\usepackage{theorem}
\theorembodyfont{\rmfamily}

\usepackage{amssymb,mathrsfs,color}
\usepackage[cmex10]{amsmath}
\usepackage{array}
\usepackage{graphicx}

\usepackage{latexsym}

\usepackage{boxedminipage}
\usepackage{fancybox}
\usepackage{ascmac}
\usepackage{amsmath}
\usepackage{cite}

\usepackage{dsfont}

\usepackage{url}

\newtheorem{proof}         {Proof}

\newtheorem{proposition}            {Proposition}

\newtheorem{remark}            {Remark}

\DeclareMathOperator*{\argmin}{arg\,min}

\newcommand{\mat}[1]{\left( \: \begin{matrix} #1 \end{matrix} \:\right)}
\newcommand{\spliteq}[1]{\begin{split} #1 \end{split}}

\begin{document}

\setlength{\abovedisplayskip}{7pt}  
\setlength{\belowdisplayskip}{7pt}  

\begin{frontmatter}

\title{
Day-Ahead Energy Market as Adjustable Robust Optimization: Spatio-Temporal Pricing of Dispatchable Generators, Storage Batteries, and Uncertain Renewable Resources
\vspace{-5mm}
\thanksref{footnoteinfo}}  

\thanks[footnoteinfo]{
Corresponding author: T.~Ishizaki, Tel. \&\ Fax: +81-3-5734-2646. 
}

\author[TIT]{Takayuki Ishizaki$^{*}$}\ead{ishizaki@sc.e.titech.ac.jp},    
\author[TUM]{Masakazu Koike}\ead{mkoike0@kaiyodai.ac.jp},
\author[TUS]{Nobuyuki Yamaguchi}\ead{n-yama@rs.tus.ac.jp},
\author[TUS]{Yuzuru Ueda}\ead{ueda@ee.kagu.tus.ac.jp},
and
\author[TIT]{Jun-ichi Imura}\ead{imura@sc.e.titech.ac.jp}   

\address[TIT]{Tokyo Institute of Technology; 2-12-1, Ookayama, Meguro, Tokyo, 152-8552, Japan.}   
\address[TUM]{Tokyo University of Marine Science and Technology; 4-5-7, Kounan, Minato, Tokyo, 108-0075, Japan.}   
\address[TUS]{Tokyo University of Science; 6-3-1, Niijuku, Katsushika, Tokyo 125-8585, Japan.}


\begin{keyword}                            
Day-ahead energy markets, Adjustable robust optimization, Spatio-temporal pricing, Distributed energy resources, Uncertainty of renewable energy resources, Convex analysis.  
\end{keyword}                              

\begin{abstract}
In this paper, we present modeling and analysis of day-ahead spatio-temporal energy markets in which each competitive player or aggregator aims at making the highest profit by managing a complex mixture of different energy resources, such as conventional generators, storage batteries, and uncertain renewable resources.
First, we develop an energy market model in terms of an adjustable robust convex program.
This market modeling is novel in the sense that the prosumption cost function of each aggregator, which evaluates 
the cost or benefit to realize an amount of spatio-temporal energy prosumption, is a multi-variable function resulting from a ``parameterized" max-min program, in which the variable of the prosumption cost function is involved as a continuous parameter and the variable of dispatchable resources is involved as an adjustable variable for energy balance.
This formulation enables to reasonably evaluate a reward for intertemporal dispatchability enhancement and a penalty for renewable energy uncertainty in a unified way.
In addition, it enables to enforce a market regulation in which every aggregator is responsible for absorbing his/her renewable energy uncertainty by managing his/her own dispatchable energy resources.
Second, in view of social economy as well as personal economy, we conduct a numerical analysis on the premise of several photovoltaic penetration levels.
In this numerical analysis, using a bulk power system model of the north east area in Japan, we demonstrate that renewable generators do not always have priority of energy supply higher than conventional generators due to their uncertainty and limited dispatchability, meaning that the merit order of conventional and renewable generators can reverse.
Furthermore, we analyze long-term evolution of competitive energy markets demonstrating that there can be found a social equilibrium of battery penetration levels, at which maximum personal profit with respect to battery system enhancement is attained.

\vspace{-3pt}
\end{abstract}\vspace{-6pt}

\end{frontmatter}


\section{Introduction}
\subsection{Research Background}

The development of a smart grid has been recognized as one of the key issues in addressing environmental and social concerns, such as the sustainability of energy resources and the efficiency of energy management \cite{annaswamy2013ieee,chu2012opportunities}.
In future power systems operation, it is crucial to appropriately manage a complex mixture of multiple types of energy resources, such as conventional generators, energy storage systems and devices, controllable and shiftable loads, and uncertain renewable resources, towards the realization of economically efficient supply-demand balance of energy.

The penetration of energy storage systems, such as electric vehicles and home energy storage systems, is generally supposed to be spatially distributed due to the limitation of installation capability \cite{dunn2011electrical}.
Photovoltaic generator installation is also supposed to be spatially distributed, especially in Japan \cite{shum2007photovoltaic}, for effective use of roof top spaces.
Even though the impact of such individual materials and components on the grid may be tiny, the aggregation of them has high potential to serve for supply-demand balance.
Thus, an aggregator, a manager of accessible energy resources, can be a strong stakeholder in electricity markets, the modeling and analysis of which are the main subject of this paper.

\subsection{Literature Review}\label{seclr}

The existing modeling and analysis of electricity markets are considerably wide-ranging in terms of objectives, settings, and concepts; see  \cite{bublitz2019survey} for a recent survey.
For a better understanding of our subject and focus, we first consider classifying existing works into the following (not necessarily mutually exclusive) categories.

\subsubsection{Day-ahead and Real-time}

The main difference between day-ahead and real-time markets can be explained in terms of ``offline scheduling" and ``online operation."
In a day-ahead market, each market player transacts their future operation schedule of  energy resources, which mainly prescribes the amounts of energy production and consumption in consideration of ancillary service for stable supply \cite{liu2016bidding,sarker2016optimal,sardou2016energy,aravena2017renewable,he2015optimal}.
On the other hand, each market player in a real-time market is supposed to transact their capability of energy resources to compensate discrepancies between operation schedules and real-time operation results \cite{knudsen2015dynamic,shiltz2016integrated,wang2016real}.
Such discrepancies are possibly caused by several factors, e.g., load prediction errors, the difference of time resolutions in scheduling and real operation, and unexpected fluctuation of renewable resources.
In general, profit-maximizing strategies in day-ahead and real-time markets are mutually correlated because of the limitation of energy resources, leading to dependency of day-ahead and real-time prices.
Optimization models considering such a market correlation are analyzed in \cite{aravena2017renewable,schneider2018energy}.

\subsubsection{Spatial and Temporal}
The notion of market transaction can be extended to space and time; see, e.g., \cite{sardou2016energy,aravena2017renewable} for models of spatio-temporal pricing.
In particular, spatial pricing in electricity markets is called \textit{locational marginal pricing} or \textit{nodal pricing} \cite{bohn1984optimal,hogan1992contract,schweppe2013spot,virasjoki2016market}, which can evaluate an impact of spatial constraints in a power network, such as the voltage capacity and thermal capacity of transmission lines.
Depending on the levels of transmission line congestion, an independent system operator (ISO) determines spatially distributed energy prices based on a marginal social cost for stable power supply.

Energy prices can also have a temporal correlation \cite{liu2016bidding,sarker2016optimal,he2015optimal,aravena2017renewable}.
This aspect deserves attention especially in recent day-ahead markets.
This is because offline scheduling is significant in considering intertemporal optimization of the mixture of different types of energy resources.
As mentioned in \cite{annaswamy2013ieee}, an electricity market that can explicitly involve such intertemporal optimization is indispensable for making use of the power shiftability of batteries or flexible loads, which will work as key energy resources for future smart grid operation.

\subsubsection{Open-loop and Closed-loop}\label{SecOC}
The behavior of market players can be modeled as an optimization problem in which a personal profit maximization is performed in consideration of market prices.
This modeling can be ``open-loop," meaning that a market price, or its prediction, is assumed to be given independently of the decision of a market player under consideration, i.e., a price-taker setting \cite{liu2016bidding,sarker2016optimal,he2015optimal,schneider2018energy,correa2018robust}.
Such an open-loop model is valid, as long as the situations of most other players, e.g., the constitution of energy resources and the strategy in profit maximization, is stationary.
However, in general, market prices are determined as assembly of the decisions of all market players.
Such a market model is ``closed-loop," meaning that the player decision is mutually correlated with the resultant price 
\cite{sardou2016energy,savelli2017optimization,aravena2017renewable,virasjoki2016market,savelli2018new,cornelusse2019community}.
This type of closed-loop modeling is crucial to discuss long-term evolution of competitive markets in which energy resource constitution is non-stationary.
An important difference between open-loop and closed-loop modeling is that the former is relevant to economic scheduling or strategic bidding from the standpoint of a ``single player" while the latter is relevant to system design or mechanism design from the standpoint of a ``system authority."

\subsubsection{Convex and Non-convex}
Microeconomics is grounded in convex analysis \cite{rockafellar1970convex,boyd2004convex}.
As thoroughly discussed in \cite{schweppe2013spot}, an optimal energy price can be calculated as a marginal social cost of economic dispatch problems, which are often formulated as convex programs.
This pricing method builds on the convexity of social costs, implying that a non-convex cost, such as the start-up cost of generators, cannot be reflected in the resultant marginal price; see, e.g., \cite{liu2016bidding,sardou2016energy,wang2016real,bertsimas2013adaptive} for models with such a unit commitment (UC) problem.
A theoretically grounded approach to dealing with such non-convexity is approximate convexification of social costs, called \textit{convex hull pricing}.
This method is shown to minimize the amount of particular side-payments caused by convexification.
However, it generally yields another complicate issue on revenue adequacy, i.e., a financial issue on revenue allocation, which is often affected by a subjective decision of ISO; see \cite{schiro2016convex} for details.

\subsubsection{Bottom-up and Top-down}

Last, but not least, the difference of market modeling concepts is explained in terms of ``bottom-up" and ``top-down."
A bottom-up approach is an engineering approach that is based on a stacking-up process of detailed elements.
Existing works can mostly be classified into this category, developing complex integrated market models.
Though such a model can serve for elaborate empirical studies, it is often cumbersome to find out key drivers and structures due to a high degree of parametric and structural freedom.
A top-down approach is a systems theoretic approach, or a reverse engineering approach, that is based on a breaking-down process starting from an ``appropriate level of abstraction."
Examples include \cite{schweppe2013spot,schiro2016convex}, which are based on supply-demand models without distributed energy resources.
A major advantage of this approach is to provide a simple but principled framework to discuss a universal law in electricity markets.
Synthesis of both bottom-up and top-down approaches would be significant to develop a principled market framework.

\subsection{Focus and Contribution of Present Paper}

With the categorization in Section~\ref{seclr}, we model and analyze an electricity market that is day-ahead, spatio-temporal, closed-loop, and convex, which we call here a \textit{spatio-temporal energy market}.
In Section~\ref{secpfm}, based on a top-down approach, we first develop a mathematical foundation for modeling and analysis of the spatio-temporal energy market, where each competitive aggregator aims at making the highest profit by managing a complex mixture of different energy resources.
The theoretical contributions in this paper are summarized as follows.
\begin{itemize}
\item  We formulate a novel prosumption cost function of each aggregator as a multi-variable function that results from a parameterized max-min program, given as the collection of an infinite number of max-min programs.
\item We prove the convexity of such a prosumption cost function, in which the amount of spatio-temporal energy prosumption is involved as a non-adjustable variable while the variable of dispatchable resources is involved as an adjustable variable for energy balance.
\end{itemize}

\begin{table*}[t]\centering
\caption{Example of market result. Three aggregators and two time spots.}\vspace{9pt}
\scriptsize
\begin{tabular}{|c||c|c|c|c|}\hline
& Agg.~1 (Producer) & Agg.~2 (Consumer) & Agg.~3 (Prosumer) & Clearing Price  \\ \hline \hline
Spot 1 (AM)& 150 {\rm [kWh]} & -250 {\rm [kWh]} & 100 {\rm [kWh]} & 10 {\rm [JPY/kWh]} \\ \hline
Spot 2 (PM)& 100 {\rm [kWh]} & -50 {\rm [kWh]} & -50 {\rm [kWh]} & 5 {\rm [JPY/kWh]} \\\hline
\end{tabular}
\label{tabex}\vspace{6pt}
\end{table*}

It will be found that the resultant market clearing problem, which is formulated as a social cost minimization problem composed of a family of the parameterized max-min programs, is given as an ``adjustable" robust convex program \cite{ben2004adjustable}, which is also referred to as a two-stage robust convex program \cite{zeng2013solving,yanikouglu2019survey}.
To the best of the authors' knowledge, such an idea of describing a ``closed-loop" energy market model based on an adjustable robust convex program has not been reported in the literature, though scheduling problems or open-loop (price taker-based) market models with uncertain resources, such as robust UC problems, have been discussed in terms of two-stage robust or stochastic optimization over recent years; see \cite{liu2016bidding,jiang2012robust,bertsimas2013adaptive,zheng2015stochastic,cobos2016least} and references therein.
As explained in Section~\ref{SecOC}, such scheduling problems are formulated from the viewpoint of a single player who can directly handle all energy resources, while our market model in this paper is formulated from the viewpoint of a system authority who considers mechanism design such that market players, each of whom has his/her own energy resources, behave as desired.
We also remark that the significance of adjustable robust optimization and difference from the standard robust optimization are clearly stated in \cite{ben2004adjustable}.

The proposed market modeling enables to evaluate not only a reward for enhancement of intertemporal dispatchability owing to dispatchable energy resources, but also a penalty for uncertainty due to renewable energy resources. 
Furthermore, it enables to enforce a ``market regulation" in which each aggregator is responsible for absorbing the uncertainty of his/her renewable power generation so that his/her dispatchable prosumption schedule can be regularly transacted in the day-ahead energy market.
These properties, deriving from the formulation based on adjustable robust optimization, give a clear distinction from the existing market models reviewed above.
In fact, the market regulation considered in this paper is compatible with the ``planned balancing policy" that has been recently published in Japan for power market liberalization, where both consumers and producers must submit deterministic consumption and generation schedules to a system operator, called Organization for Cross-regional Coordination of Transmission Operators (OCCTO); see, e.g., \cite[Section~19]{kobayashi2019electricity} for details. 
We remark that our market modeling based on adjustable robust optimization can be seen as a mathematical interpretation of the planned balancing policy in Japan.

In Section~\ref{secnum}, on the premise of several photovoltaic penetration levels, we conduct a numerical analysis to investigate what level of battery penetration can be realized as a result of aggregators' rational decisions, under the market regulation in which each aggregator is responsible for absorbing the uncertainty of his/her renewable power generation.
The specific contribution there is to give the following insights from the viewpoint of mechanism design.
\begin{itemize}
\item 
The proposed market model in terms of adjustable robust optimization can naturally reproduce the empirical finding \cite{lunackova2017merit} that renewable generators do not always have priority of energy supply higher than conventional generators due to their uncertainty and limited dispatchability.
\item There is a social equilibrium of battery penetration levels, at which maximum personal profit with respect to battery system enhancement is attained.
\end{itemize}
The first insight is relevant to the ``merit order" of conventional versus renewable generators.
The second is relevant to long-term evolution of competitive energy markets; see Section~\ref{seciq} for more details of motivation to study.
This paper builds on preliminary versions \cite{ishizaki2017distributed,ishizaki2017bidding}.
Compared with them, this paper newly gives the formulation of uncertain renewable resources, which enables the analysis of battery penetration levels.
In addition, spatial constraints are also considered in the proposed model.

\section{Preliminaries}

\subsection{Day-Ahead Energy Market as Convex Optimization}\label{secmkopt}

In this subsection, to make the following discussion self-consistent, we overview the relation between a day-ahead energy market and an optimization problem from a top-level view of abstraction.
For clarity, we first consider a small example where three market players, called aggregators, participate in a day-ahead market.
In particular, the schedule of the total energy amount in the forenoon and that in the afternoon are supposed to be transacted, meaning that the whole day schedule transaction is divided into two energy markets at the AM and PM spots.

The objective of ISO is to find a suitable set of transaction energy amounts among the aggregators and a set of clearing prices assigned to the transaction energy amounts.
For example, suppose that a market result is determined as in Table~\ref{tabex}, where positive and negative energy amounts correspond to production and consumption, respectively.
In this example, the transaction energy amounts can be written as
\[
x_{1}^{*}=
\mat{150\\
100},
\quad
x_{2}^{*}=\mat{
-250\\
-50
}, \quad
x_{3}^{*}=\mat{
100\\
-50
}
\]
and the clearing price can be written as 
\[
\lambda^{*}=\mat{
10\\
5
}.
\]
Note that all these vectors, which are the decision variables of ISO, are two-dimensional vectors, the dimension of which corresponds to the number of time spots.
Furthermore, the transaction energy amounts are balanced, i.e., 
\[
x_{1}^{*}+x_{2}^{*}+x_{3}^{*}=0.
\]
This represents the balance of production and consumption (prosumption) energy amounts at ``every" time spot.

Let $\mathcal{A}$ denote the label set of aggregators and let $\mathcal{T}$ denote the label set of time spots on the day of interest.
In this paper, the event of finding a set of balanced prosumption profiles and a clearing price profile, denoted by $(x_{\alpha}^{*})_{\alpha\in \mathcal{A}}$ and $\lambda^{*}$, is referred to as \textit{market clearing}; A specific example in which 
\[
\mathcal{A}=\{1,2,3\},\quad \mathcal{T}=\{{\rm AM},{\rm PM}\}
\]
is given above.
With this terminology, we review a general market clearing problem in terms of intertemporal optimization.
To this end, we introduce the notion of a \textit{profit}, which is a personal objective function of aggregators.
Under a clearing price profile $\lambda^{*}$, the resultant profit of the $\alpha$th aggregator is written as
\begin{equation}\label{profJ}
J_{\alpha}(x_{\alpha}^{*};\lambda^{*})=\langle\lambda^{*},x_{\alpha}^{*}\rangle-F_{\alpha}(x_{\alpha}^{*})
\end{equation}
where the inner product $\langle\lambda^{*},x_{\alpha}^{*}\rangle$ corresponds to the income obtained from the energy market and $F_{\alpha}$ denotes the cost function to realize a prosumption profile.
Note that $F_{\alpha}$ may be negative, meaning that it may be a benefit function depending on its sign.
Similarly, the income may also be negative, meaning that the inner product term can be regarded as income or outgo depending on its sign.
For instance, 
\[
\langle\lambda^{*},x_{1}^{*}\rangle=2000,\quad\langle\lambda^{*},x_{2}^{*}\rangle=-2750,\quad
\langle\lambda^{*},x_{3}^{*}\rangle=750
\]
in the three-aggregator example above.

The social profit, which is a social objective function of ISO, is given as
\begin{equation}\label{sosprof}
\displaystyle 
\sum_{\alpha = 1}^{|\mathcal{A}|}
J_{\alpha}(x_{\alpha}^{*};\lambda^{*})=\left\langle\lambda^{*},
\sum_{\alpha = 1}^{|\mathcal{A}|}
x_{\alpha}^{*}\right\rangle-
\sum_{\alpha = 1}^{|\mathcal{A}|}
F_{\alpha}(x_{\alpha}^{*}).
\end{equation}
Note that the social income, represented as the inner product term, is zero as long as the balance of prosumption profiles is attained.
From this observation, we see that the social profit maximization problem is equivalent to the social cost minimization problem formulated as
\begin{equation}\label{marketp}
\displaystyle \min_{
(x_1,\ldots,x_{|\mathcal{A}|})
\in \mathcal{C}_1 \times\cdots \times \mathcal{C}_{|\mathcal{A}|}
}\hspace{0pt}
\sum_{\alpha = 1}^{|\mathcal{A}|}
F_{\alpha}(x_{\alpha})\ \ {\rm s.t.}\ \ 
\sum_{\alpha = 1}^{|\mathcal{A}|}
x_{\alpha}=0.
\end{equation}
In this optimization problem, a feasible domain $\mathcal{C}_{\alpha}$ is introduced for each prosumption profile $x_{\alpha}$.
In particular, we define it as
\begin{equation}\label{conC}
\mathcal{C}_{\alpha}:= \{x_{\alpha} : |x_{\alpha}| \leq C_{\alpha} \},
\end{equation}
which represents the transmission line constraint on $x_{\alpha}$.
For a better understanding, we give a schematic depiction of the three-aggregator example in Fig.~\ref{figthreeagg}\textbf{(a)} where $x_{\alpha}$ can be seen as the prosumption profile outflowing from the $\alpha$th aggregator to the central bus at which supply-demand balance is attained.
This is a single bus system, meaning that no specific network structure among aggregators is considered.

It is clear that an optimal primal variable $(x_{\alpha}^{*})_{\alpha\in \mathcal{A}}$ corresponds to the transacted prosumption profiles.
Furthermore, an optimal dual variable $\lambda^{*}$, i.e., the Lagrange multiplier associated with the prosumption balance constraint, corresponds to the clearing price profile \cite{bohn1984optimal,schweppe2013spot}.
Thus, the socially optimal market clearing problem is equivalent to the social cost minimization problem \eqref{marketp}, provided that it is a convex program \cite{rockafellar1970convex,boyd2004convex}.
In addition, the socially optimal market results can attain the profit maximization of each aggregator; 
see Proposition~\ref{proppc} for details.

\begin{figure}[t]
\centering
\includegraphics[width=60mm]{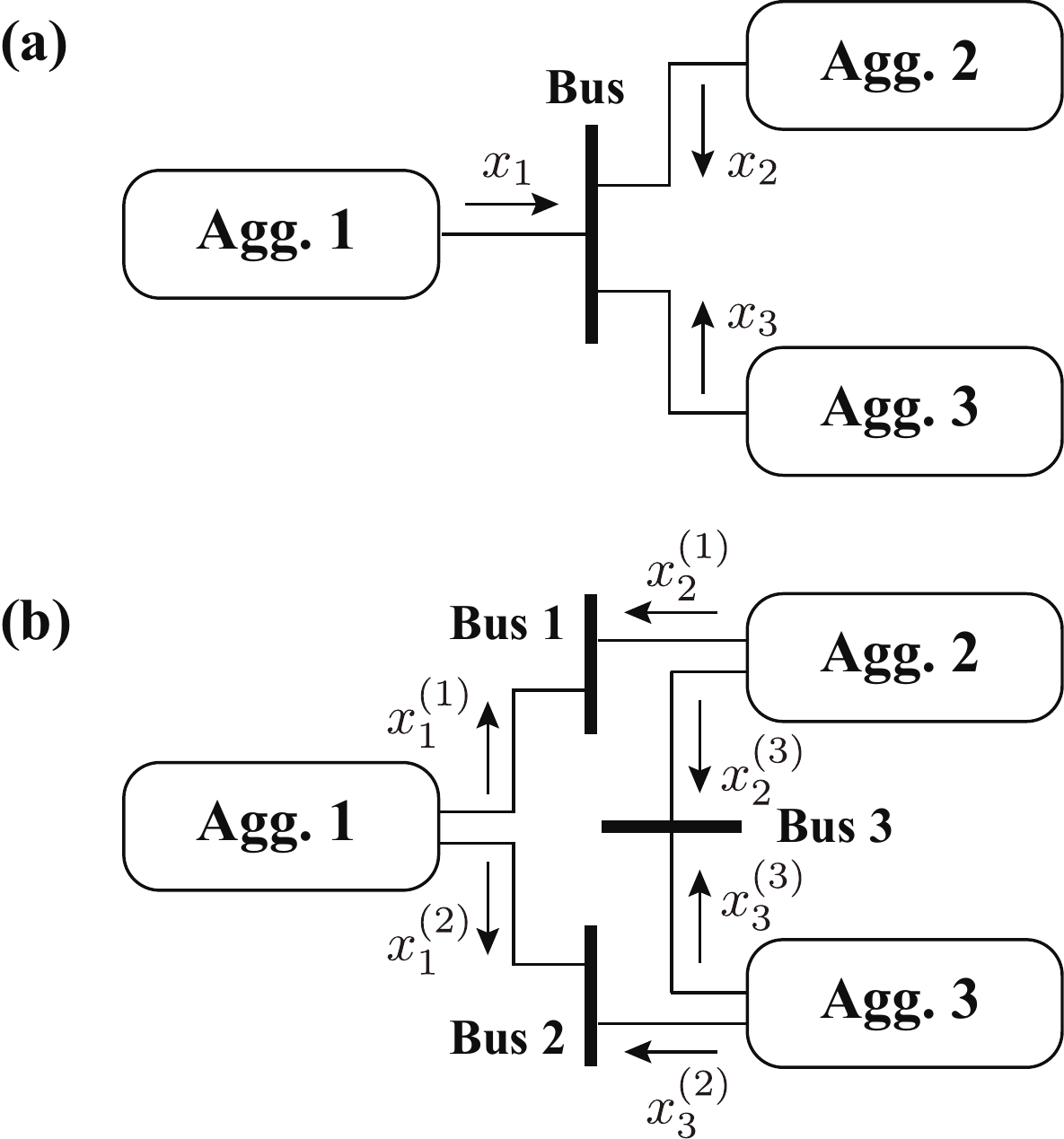}
\vspace{3pt}
\caption{
\textbf{(a)} Single-bus system.
\textbf{(b)} Three-bus system.
}
\label{figthreeagg}
\vspace{8pt}
\end{figure}

\subsection{Day-Ahead Energy Market with Spatial Constraints}\label{secspatial}

In this subsection, a network structure among aggregators is introduced.
For explanation, we again refer to the three-aggregator example in Section~\ref{secmkopt}.
Let us give a network structure among three aggregators as in Fig.~\ref{figthreeagg}\textbf{(b)}, where three intermediate buses are introduced.
The prosumption profile outflowing from the $\alpha$th aggregator to the $i$th bus is denoted by $x^{(i)}_{\alpha}$.
In a manner similar to \eqref{conC}, its transmission line constraint can be written with 
\[
\mathcal{C}_{\alpha}^{(i)}:= \{x_{\alpha}^{(i)} : |x_{\alpha}^{(i)}| \leq C_{\alpha}^{(i)} \}.
\]
If there is no direct transmission line between the $\alpha$th aggregator and the $i$th bus, we set $C_{\alpha}^{(i)}=0$, which can enforce $x_{\alpha}^{(i)}=0$.
For example, we set
\[
C_1^{(3)}=0, \quad C_2^{(2)}=0, \quad C_3^{(1)}=0
\]
as being compatible with Fig.~\ref{figthreeagg}\textbf{(b)}.

With slight abuse of notation, we redefine the variable of the $\alpha$th aggregator as the augmented prosumption profile
\begin{equation}\label{augpro}
x_{\alpha}:=(x^{(1)}_{\alpha},\ldots,x^{(|\mathcal{B}|)}_{\alpha})
\end{equation}
where $\mathcal{B}$ denotes the label set of buses.
Note that this $x_{\alpha}$ represents a ``spatio-temporal" prosumption profile, which is $(|\mathcal{T}| \times |\mathcal{B}|)$-dimensional.
Then, we can compactly represent the set of all transmission line constraints on $x^{(1)}_{\alpha},\ldots,x^{(|\mathcal{B}|)}_{\alpha}$ by $x_{\alpha} \in \mathcal{C}_{\alpha}$ where $\mathcal{C}_{\alpha}$ is also redefined as
\begin{equation}\label{augcalC}
\mathcal{C}_{\alpha}:= \mathcal{C}_{\alpha}^{(1)}\times \cdots \times \mathcal{C}_{\alpha}^{(|\mathcal{B}|)}.
\end{equation}
Based on this augmented representation, the market clearing problem with the spatial constraints can also be written in the form of \eqref{marketp}.
It should be remarked that the equality constraint of \eqref{marketp} is equivalent to
\[
\sum_{\alpha = 1}^{|\mathcal{A}|}
x_{\alpha}^{(i)}=0,\quad
\forall i \in \mathcal{B}.
\]
In accordance with each balancing equation, a clearing price profile, denoted by $\lambda^{(i)}$, is assigned at each of all buses.
In other words, the resultant clearing price profile is also a spatio-temporal variable as
\begin{equation}\label{auglam}
\lambda := (\lambda^{(1)},\ldots,\lambda^{(|\mathcal{B}|)}),
\end{equation}
which is $(|\mathcal{T}| \times |\mathcal{B}|)$-dimensional.

\subsection{Key Questions}\label{seciq}

In this subsection, we list several key questions to be discussed in this paper.
The first question is:
\begin{description}\vspace{1pt}
\item[\textbf{Q1}:] How to model the cost function of prosumption profiles in a reasonable manner?
\end{description}\vspace{1pt}
In the discussion of Section~\ref{secmkopt} above, we implicitly suppose that the prosumption cost function $F_{\alpha}$ in \eqref{marketp} is a given convex function, which must be ``multi-variable," i.e., 
\[
F_{\alpha}:\mathbb{R}^{|\mathcal{T}| \times |\mathcal{B}| }\rightarrow \mathbb{R}.
\]
However, in practice, it is not always easy to identify such an ``spatio-temporally" correlated cost function.
In addition, aggregators are generally ``prosumers," meaning that the prosumption profile $x_{\alpha}$ should be a complex mixture of different types of energy resources.
A prosumption cost function should reflect idiosyncratic properties and structures of each prosumer.
An answer to this question will be given in Section~\ref{secpfm}.

The second question is:
\begin{description}\vspace{1pt}
\item[\textbf{Q2}:] Is it reasonable to say in the spatio-temporal energy market that uncertain renewable resources have priority higher than dispatchable generators in terms of conventional merit order?
\end{description}\vspace{1pt}
It is often mentioned that renewable generators have priority higher than conventional generators because the marginal cost of renewable resources is zero, or close to zero.
However, this conclusion is possibly premature in the spatio-temporal energy market.
This is because the marginal cost function to realize a prosumption profile $x_{\alpha}$ is not only ``multi-variable" but also ``multi-valued," i.e.,
\begin{equation}\label{mulpF}
\partial F_{\alpha}:\mathbb{R}^{|\mathcal{T}| \times |\mathcal{B}| }\rightarrow \mathbb{R}^{|\mathcal{T}| \times |\mathcal{B}|}
\end{equation}
which is called the \textit{subdifferential} of $F_{\alpha}$ shown to be a monotone mapping \cite{rockafellar1970convex}.
This multi-dimensionality comes from the spatio-temporal correlation of prosumption cost functions.
In addition, renewable and conventional generators should be distinguished in terms of the difference of realizable prosumption profiles, i.e., the difference of intertemporal dispatchability.
Thus, it is not trivial to determine the ``merit order" of different energy resources.

Furthermore, the uncertainty of renewable power generation must be considered in discussing its marginal cost.
Note that, though the day-ahead prediction of renewable power generation necessarily involves some degree of uncertainty, the prosumption profile $x_{\alpha}^{*}$ should be ``deterministic" so that it can be regularly transacted in the day-ahead market.
This implies that each aggregator should be responsible for absorbing the uncertainty of his/her renewable power generation.
In general, large uncertainty can lead to conservative management of dispatchable energy resources; thereby reducing social profit as well as personal profit of each aggregator.

One prospective way to regulating the uncertain fluctuation of renewable power generation would be a smart use of energy storage systems.
In fact, its potential has been attracting much attention towards effective integration of dispatchable and renewable power generation.
However, the existing works have not yet been paid much attention on the question:
\begin{description}\vspace{1pt}
\item[\textbf{Q3}:] What level of energy storage penetration is economically reasonable from the viewpoints of not only social profit but also personal profit?
\end{description}\vspace{1pt}
In general, large penetration of energy storage systems serves for improving social profit because it has a potential to reduce the waste of renewable power curtailments and to regulate load profiles by peak load shifting.
However, from the viewpoint of aggregator's ``personal" profit, it is not clear what degree of increased profit deserves to make an investment in the enhancement of energy storage systems.
We will discuss \textbf{Q2} and \textbf{Q3} through the numerical analysis in Section~\ref{secnum}.

\section{Mathematical Foundation}\label{secpfm}

\subsection{Aggregator Model}\label{secagm}

In this subsection, we provide a model of aggregators as incorporating their idiosyncrasy.
This is a breaking-down process to elaborate on the abstracted market model in \eqref{marketp}.
For simplicity of notation, we do not make a distinction among the aggregators in this model description, i.e., we drop the subscript $\alpha$ as long as there is no chance of confusion.

The prosumption energy of an aggregator at the $t$th time spot can be described as
\begin{equation}\label{balagg}\textstyle
\sum_{i=1}^{|\mathcal{B}|}x^{(i)}_{t}
=g_{t}-l_{t}+p_{t}-q_{t}+\eta^{\rm out}\delta^{\rm out}_{t}-\frac{1}{\eta^{\rm in}}\delta^{\rm in}_{t},\quad t\in \mathcal{T}
\end{equation}
where 
$x_{t}^{(i)} \in \mathbb{R}$ denotes the prosumption energy outflowing to the $i$th bus,
$g_{t}\in \mathbb{R}_{+}$ denotes the power generation of dispatchable generators, 
$l_{t}\in \mathbb{R}_{+}$ denotes the load, 
$p_{t}\in \mathbb{R}_{+}$ and $q_{t}\in \mathbb{R}_{+}$ denote the power generation and  power curtailment of renewable generators, and 
$\delta^{\rm in}_{t}\in \mathbb{R}_{+}$ and 
$\delta^{\rm out}_{t}\in \mathbb{R}_{+}$ denote the battery charge and discharge energy.
The constants $\eta^{\rm in}$ and $\eta^{\rm out}$ denote the charge and discharge efficiency, respectively, each of which takes a value in $(0,1]$.
Note that the sign of $x_{t}^{(i)}$ is positive for outflow direction to the $i$th bus.

For a better understanding of the battery terms, let us consider the case where the aggregator of interest is just composed of a battery system, and is connected to only the $i$th bus.
In this case, \eqref{balagg} is reduced to
\[\textstyle
x_t^{(i)}=\eta^{\rm out}\delta^{\rm out}_{t}-\frac{1}{\eta^{\rm in}}\delta^{\rm in}_{t},\quad t\in \mathcal{T}.
\]
Suppose that both $\eta^{\rm out}$ and $\eta^{\rm in}$ are 0.9.
Then, if the aggregator sells $x_t^{(i)}=9$, this energy is supplied by $\delta^{\rm out}_{t}=10$ and $\delta^{\rm in}_{t}=0$.
Note that the decrement in battery-stored energy is larger than the sold energy due to discharging loss.
On the other hand, if the aggregator sells $x_t^{(i)}=-10$, which implies purchase, this energy is consumed by $\delta^{\rm out}_{t}=0$ and $\delta^{\rm in}_{t}=9$, meaning that the increment in battery-stored energy is smaller than the purchased energy due to charging loss.
Note that the increment in the battery-stored energy at the $t$th time spot can be written as $\delta^{\rm in}_{t} - \delta^{\rm out}_{t}$.

In the following, we denote the stacked vector of a symbol, called a profile, by that without the subscript $t$.
For example, the load profile $l = (l_t)_{t \in \mathcal{T} }$ represents the sequence of load amounts, which is $|\mathcal{T}|$-dimensional.
Furthermore, we denote by $x$ in \eqref{augpro} the collection of the prosumption profiles outflowing from the aggregator, called the \textit{spatio-temporal prosumption profile}.
Throughout the paper, the load profile $l$ is supposed to be a given constant vector, i.e., the prediction of the load profile is supposed to perform without a prediction error.
On the other hand, the renewable power generation profile $p$ is supposed to be an uncertain variable that can vary within a scenario set, denoted by $\mathcal{P}$.
In particular,
\begin{equation}\label{PVsc}
p\in \mathcal{P},\quad\mathcal{P}:=\{p^{(1)},\ldots,p^{(m)}\},
\end{equation}
where $m$ different scenarios are considered in the day-ahead prediction.
For simplicity of discussion, we here suppose that $\mathcal{P}$ is composed of a finite number of renewable scenarios, though the following discussion can be extended also to the case of continuous sets of infinite many scenarios.
Note that, if necessary, we can also model the uncertainty of load profiles in the same way as that of renewable profiles.

The dispatchable power generation profile $g$, the renewable power curtailment  profile $q$, and the battery charge and discharge power profiles $\delta^{\rm in}$ and $\delta^{\rm out}$ are  decision variables that can be controlled by the aggregator.
To realize a desired prosumption profile $x$, the aggregator aims at controlling $g$,  $q$, and $\delta:=(\delta^{\rm in},\delta^{\rm out})$ in compliance with some physical constraints.
In particular, the bounds for dispatchable power generation and the limitation of inverter and battery capacities are represented by given domains $\mathcal{G}$ and $\mathcal{D}$, i.e.,
\begin{equation}\label{gddom}
g\in \mathcal{G},\quad\delta\in \mathcal{D}.
\end{equation}
For example, these can represent a constraint $\underline{g}\leq g\leq\overline{g}$, where $\underline{g}$ and $\overline{g}$ are given constant vectors representing the lower and upper bounds for generator outputs.
The renewable power curtailment profile must not be greater than the the renewable power generation profile, i.e., 
$0\leq q\leq p$ for each scenario $p\in \mathcal{P}$.
In this paper, this is written by
\begin{equation}
q\in \mathcal{Q}(p),\quad p\in \mathcal{P}.
\end{equation}
Note that the domain $\mathcal{Q}(p)$ is dependent on which renewable power generation profile $p$ arises from the scenario set $\mathcal{P}$.

\subsection{Characterization of Realizable Prosumption Profiles}

We consider giving a mathematical description of intertemporal dispatchability of aggregator's energy prosumption.
With respect to each spatio-temporal prosumption profile $x$ under a renewable scenario $p\in \mathcal{P}$, we denote the feasible domain of $g$,  $\delta$, and $q$, which are called the dispatchable variables, by
\begin{equation}\label{calfd}
\mathcal{F}(x;p):=\bigl\{(g,\delta,q)\in \mathcal{G}\times \mathcal{D}\times \mathcal{Q}(p):\mbox{\eqref{balagg} holds}\bigr\}.
\end{equation}
Note that $x$ and $p$ are involved as ``parameters" in the constraint of the dispatchable variables.
More specifically, a feasible domain $\mathcal{F}(x;p)$ of the stacked version $(g,\delta,q)$ of the dispatchable variables is determined as a standard set if the parameters $x$ and $p$ are fixed as constant vectors.

Based on this parameterized domain, we define the following set of realizable prosumption profiles.

\begin{proposition}\label{propdom}
Suppose that $\mathcal{G}$ and $\mathcal{D}$ are convex.
Then the set of realizable prosumption profiles defined by
\begin{equation}\label{realsp}
\mathcal{X}:=
\bigl\{
x\in \mathbb{R}^{|\mathcal{T}| \times |\mathcal{B}| }
:\mathcal{F}(x;p)\neq\emptyset,\ \forall p\in \mathcal{P}\bigr\}
\end{equation}
is convex.
\end{proposition}

\begin{proof}
We see that $\mathcal{X}$ can be written as
\[
\mathcal{X}=\bigcap_{p\in \mathcal{P}}\mathcal{X}'(p),\quad\mathcal{X}'(p):=\bigl\{x\in \mathbb{R}^{|\mathcal{T}| \times |\mathcal{B}|}:\mathcal{F}(x;p)\neq\emptyset\bigr\}.
\]
We can prove the convexity of $\mathcal{X}'(p)$ as follows.
The feasible domain of  $\sum_{i=1}^{|\mathcal{B}|}x^{(i)}$, denoted by $\mathcal{X}''(p)$, is convex because it is an affine mapping of the convex domain $\mathcal{G}\times \mathcal{D}\times \mathcal{Q}(p)$.
As shown in \cite[Theorem~3.4]{rockafellar1970convex}, the inverse image of a convex domain under a linear mapping is convex.
This implies that $\mathcal{X}'(p)$ is convex because it is an inverse image of $\mathcal{X}''(p)$ under summation.
Therefore, $\mathcal{X}$ is convex because it is an intersection of convex domains.
\hfill $\blacksquare$
\end{proof}

For explanation, let us consider the meaning of ``$0\in \mathcal{X}$."
This means that the supply-demand balance inside the aggregator, i.e., $x=0$, or equivalently,
\begin{equation}\label{xeq0}
x^{(1)}= \cdots = x^{(|\mathcal{B}|)}=0,
\end{equation}
is surely realizable for ``any" renewable scenario $p$ arising from $\mathcal{P}$.
This can be seen as follows.
Suppose that the $s$th renewable scenario $p^{(s)}$ arises.
Then, the condition of $\mathcal{F}(0;p^{(s)})\neq\emptyset$ means that there exists at least one combination of the dispatchable variables such that the supply-demand balance is attained in the $s$th renewable scenario, i.e.,
\[
\spliteq{
\exists (g,\delta, q) &\in\mathcal{G} \times \mathcal{D} \times \mathcal{Q}(p^{(s)}) \quad {\rm s.t.} \\
& \textstyle{ 0 =g-l+p^{(s)}-q+\eta^{\rm out}\delta^{\rm out}-\frac{1}{\eta^{\rm in}}\delta^{\rm in}}.
}
\]

Therefore,  $\mathcal{F}(0;p)\neq\emptyset$ for all  $p\in \mathcal{P}$
means that, for each of all renewable scenarios, there exists at least one combination of the dispatchable variables such that the supply-demand balance inside the aggregator is attained, i.e., \eqref{xeq0} holds.
Note that $(g,\delta,q)$ is an ``adjustable" variable, meaning that it can vary such that the energy balance in \eqref{balagg} is attained for each of all scenarios $p\in \mathcal{P}$ under a spatio-temporal prosumption profile $x\in \mathcal{X}$.

As seen here, $\mathcal{X}$ can be understood as the collection of all possible spatio-temporal prosumption profiles that can be  surely realizable for each of all renewable scenarios.
In the rest of this paper, we suppose that each aggregator transacts only a prosumption profile $x$ involved in $\mathcal{X}$.
This corresponds to a ``market regulation" in which every aggregator is responsible for absorbing the uncertainty of his/her renewable power generation by managing his/her own dispatchable energy resources.
Proposition~\ref{propdom} shows that this market regulation of uncertainty absorption ensures the convexity of $\mathcal{X}$, crucial for marginal cost pricing.

\begin{remark}
The idiosyncrasy of $\mathcal{X}$ can be understood as the individual characteristics of intertemporal dispatchability that can be produced from energy resource management.
For a better understanding, let us consider a simple aggregator that is connected to only one bus and just has a conventional generator, i.e., $x=g$.
Then, it is easy to see that $\mathcal{X}=\mathcal{G}$.
On the other hand, for an aggregator that is connected to only one bus and just has a renewable energy resource, i.e., $x=p-q$, we see that
\begin{equation}\label{pvX}
\displaystyle \mathcal{X}=
\{
x\in \mathbb{R}^{|\mathcal{T}| }:0\leq x\leq p_{{\rm min}}\},\quad p_{{\rm min}}:=\min \mathcal{P}
\end{equation}
where the minimum of $\mathcal{P}$ is taken in the element-wise sense.
Note that $p_{{\rm min}}$ represents the lower envelope of all possible renewable scenarios.
This means that the uncertainty is removed just by curtailment; 
The larger the uncertainty is, the smaller the domain $\mathcal{X}$ is.
Such a waste of energy curtailment can be reduced, e.g., if the aggregator has a battery.
This battery introduction can actually enlarge the domain of $\mathcal{X}$.
In this sense, the ``size" of $\mathcal{X}$ can be understood as a level of intertemporal dispatchability, or flexibility, of the corresponding aggregator under the market regulation of uncertainty absorption.
\end{remark}

\subsection{Deduction of Prosumption Cost Function}

We consider assessing a cost of each spatio-temporal prosumption profile $x\in \mathcal{X}$.
This gives an answer to \textbf{Q1} in Section~\ref{seciq}.
To this end, we introduce the cost functions of dispatchable power generation and battery charge and discharge as
\begin{equation}\label{gbcos}
G:\mathcal{G}\rightarrow \mathbb{R},\quad D:\mathcal{D}\rightarrow \mathbb{R},
\end{equation}
which represent practically accessible cost functions, such as the fuel cost of generators and the deterioration cost of batteries.
With this notation, we can deduce a convex prosumption cost function as follows.

\begin{proposition}\label{thmconv}
Suppose that the generation cost function $G$ and the battery cost function $D$ are convex over convex domains $\mathcal{G}$ and $\mathcal{D}$.
Then, the prosumption cost function 
\begin{equation}\label{pcosfun}
F(x):=\displaystyle \max_{p\in \mathcal{P}}\min_{(g,\delta,q)\in \mathcal{F}(x;p)}\Bigl\{G(g)+D(\delta)\Bigr\}
\end{equation}
is convex over the convex domain $\mathcal{X}$.
\end{proposition}

\begin{proof}
We see that $\sum_{i=1}^{|\mathcal{B}|}x^{(i)}$ in \eqref{balagg} is given as an affine mapping of $(g,\delta,q)$, and $x$ is an inverse image of $\sum_{i=1}^{|\mathcal{B}|}x^{(i)}$ under summation.
As shown in \cite[Theorem~5.7]{rockafellar1970convex}, the image of a convex function under an affine mapping as well as the inverse image of a convex function under an affine mapping are both convex functions.
This implies that $F'$ defined as 
\begin{equation}\label{Fprime}
F'(x;p):=\displaystyle \min_{(g,\delta,q)\in \mathcal{F}(x;p)}\Bigl\{G(g)+D(\delta)\Bigr\}
\end{equation}
is convex with respect to $x$ for each $p\in \mathcal{P}$.
In addition, the pointwise supremum of an arbitrary collection of convex functions is convex, as shown in  \cite[Theorem~5.5]{rockafellar1970convex}.
Thus, the convexity of $F$ is proven for any set $\mathcal{P}$, which may be discrete or continuous. 
\hfill $\blacksquare$
\end{proof}

The meaning of the prosumption cost function can be explained as follows.
For clarity, we first discuss the value of $F(0)$, i.e., the cost to realize the spatio-temporal prosumption profile $x=0$.
Suppose that the $s$th renewable scenario $p^{(s)}$ arises.
Then, $\mathcal{F}(0;p^{(s)})$, defined as in \eqref{calfd}, represents the set of all possible combinations of $g\in \mathcal{G}$,  $\delta\in \mathcal{D}$,  and $q\in \mathcal{Q}(p^{(s)})$ such that the supply-demand balance is attained inside the aggregator.
Note that there can be infinite many combinations of the dispatchable variables satisfying this supply-demand balance.
Thus, we can use this remaining degree of freedom to minimize the sum of the generation and battery costs, as in the minimization part of \eqref{pcosfun}.

Let us denote the minimum value by $F'(0;p^{(s)})$, for which $F'$ is defined as in \eqref{Fprime}.
Note that both $x$ and $p$ in \eqref{Fprime} are involved as parameters in the constraint of the minimization problem.
This means that $(g,\delta,q)$ is an ``adjustable" variable.
Therefore, the minimizer $(g^*,\delta^*,q^*)$ is a function dependent on $p\in \mathcal{P}$ and $x\in \mathcal{X}$.
The value of $F'(0;p^{(s)})$ corresponds to the minimum cost to attain the supply-demand balance under the specific renewable scenario $p^{(s)}$.
It should be noted that simultaneous charging and discharging of batteries is to be automatically avoided as a result of battery cost minimization.
This is because such a redundant operation causes unnecessary energy loss unless both charge and discharge efficiencies are one.

Finally, the value of $F(0)$ is defined as the worst cost among all possible renewable scenarios, i.e.,
\[
F(0)=\max_{s \in\{1,\ldots,m\}}F'(0;p^{(s)}).
\]
As seen here, to find the value of $F(0)$, or, more generally, to find the value of $F(x)$ for a ``fixed" parameter $x$, is a single max-min program.
Note that this max-min program is different from a standard robust optimization problem \cite{ben1998robust} defined as a ``min-max" program for minimizing the worst-case cost.
It should be further noted that the worst-case scenario $p^{(s^*)} \in \mathcal{P}$ is a function dependent on $x \in \mathcal{X}$.
This means that the worst-case scenario is not unique in general, but it depends on how much amounts of energy are bought or sold at the respective time spots.

The continuous function $F(x)$ can be calculated as the collection of all minimum costs to realize respective prosumption profiles, provided that each $x$ is an element of $\mathcal{X}$.
Note that the closed form of $F$ cannot be written down in general; 
It encapsulates complex information of different energy resources, such as the uncertainty of renewable power generation and the physical constraints of generators and batteries.


\begin{remark}
One may think that the convexity of $F$ is not surprising because of assuming the convexity of $G$ and $D$.
For emphasis, we remark that the prosumption cost function is defined based on a ``parameterized'' max-min program in which the variable $x$ is involved as a ``continuous parameter" in the equality constraint \eqref{balagg}.
In fact, the max-min program in the right-hand side of \eqref{pcosfun} can be reduced to a set of scenario-wise convex programs if the parameter $x$ is fixed.
In other words,  $F$ is defined as the collection of an infinite number of scenario-wise convex programs.
The significance of Proposition~\ref{thmconv} is to show, by virtue of convex analysis theory \cite{rockafellar1970convex}, that such a complicated function is convex with respect to the multi-valued parameter $x$.
It is interesting to note that, even though the maximization with respect to $p$ may be a non-convex program for a continuous set $\mathcal{P}$, being possibly non-convex, the resultant function $F$ is indeed convex with respect to $x$.
To the best of the authors' knowledge, such an idea of formulating the cost function, or, at the same time, benefit function, based on a parameterized max-min program has not been reported in the literature of spatio-temporal energy market modeling.
\end{remark}

\begin{remark}
In Proposition~\ref{thmconv}, the worst case of renewable scenarios is considered.
This can be replaced with the expectation with respect to renewable scenarios as
\begin{equation}\label{pcosfun2}
\tilde{F}(x):=\displaystyle \sum_{s=1}^{m}\pi^{(s)}
\min_{(g,\delta,q)\in \mathcal{F}(x;p^{(s)})}\Bigl\{G(g)+D(\delta)\Bigr\}
\end{equation}
where $\pi^{(s)}$ denotes the realization probability of the renewable scenario $p^{(s)}$.
This $\tilde{F}$ is also shown to be convex.
Note that the variable $x$ is involved here as a parameter of the expected cost minimization problem in the right-hand side of \eqref{pcosfun2}.
Even though this $\tilde{F}$ is defined in terms of expectation, every prosumption profile $x\in \mathcal{X}$ is surely realizable for all possible renewable scenarios $p\in \mathcal{P}$, because $\mathcal{X}$ is defined as the robust feasible domain   in \eqref{realsp}.
This is categorized into a two-stage stochastic programing framework.  
\end{remark}

\begin{remark}
We give a remark on the cost or penalty of renewable power curtailment.
In Proposition~\ref{thmconv}, we do not explicitly evaluate the cost of $q$.
However, this does not mean that the renewable power curtailment results in no loss of profit.
For clarity, we consider an aggregator that is connected to only one bus and just has a renewable resource, i.e., $x=p-q$.
Then, 
\[
F(x)=0,\quad \forall x\in \mathcal{X}
\] 
with $\mathcal{X}$ given as in \eqref{pvX}.
This means that every $x$ involved in $\mathcal{X}$ is equally realizable without causing any cost.
Consider the personal profit maximization with $J$ defined as in \eqref{profJ}.
Whenever the resultant clearing price profile is positive, i.e., $\lambda^{*}>0$, the profit-maximizing prosumption profile is found as $x^{*}=p_{{\rm min}}$.
This shows that the least renewable power curtailment is naturally obtained as the rational decision in accordance with the premise of profit maximization.
\end{remark}

\subsection{Spatio-Temporal Energy Market Model}

Because each prosumption cost function is convex as shown in Proposition~\ref{thmconv}, the socially optimal market clearing problem \eqref{marketp} is found to be an adjustable, or two-stage, robust convex program. 
In particular, \eqref{marketp} is a first-stage problem to find the optimal non-adjustable variable $(x_{\alpha}^{*})_{\alpha\in \mathcal{A}}$, while the family of
\[
F_{\alpha}(x_{\alpha})=\displaystyle \max_{p_{\alpha}\in \mathcal{P_{\alpha}}}\min_{(g_{\alpha},\delta_{\alpha},q_{\alpha})\in \mathcal{F}_{\alpha}(x_{\alpha};p_{\alpha})}\Bigl\{G_{\alpha}(g_{\alpha})+D_{\alpha}(\delta_{\alpha})\Bigr\}
\]
for all aggregators is the family of second-stage problems to find the optimal adjustable variables $(g_{\alpha}^*,\delta_{\alpha}^*,q_{\alpha}^*)_{\alpha\in \mathcal{A}}$ such that the uncertainty of renewable resources is absorbed inside individual aggregators on their own responsibility.

Based on the convexity of the market model, we can see a clear relation between microeconomics and convex analysis.
To see this, using $\mathcal{C}_{\alpha}$ in \eqref{augcalC}, which denotes the feasible domain with respect to the transmission line constraints, we consider the Lagrangian dual problem
\begin{equation}\label{lagprob}
\displaystyle \max_{\lambda}
\sum_{\alpha= 1}^{|\mathcal{A}|} 
\min_{
x_{\alpha} \in \mathcal{C}_{\alpha}
}\hspace{0pt}
\Bigl\{F_{\alpha}(x_{\alpha})-\langle\lambda,x_{\alpha} \rangle\Bigr\},
\end{equation}
which is equivalent to \eqref{marketp} owing to the strong duality.
Note that this $\lambda$ is found as a spatio-temporal variable in \eqref{auglam}, for which the inner product term is calculated as
\[
\langle\lambda,x_{\alpha} \rangle = \sum_{i=1}^{| \mathcal{B} |} \langle\lambda^{(i)},x^{(i)}_{\alpha} \rangle.
\]
As shown in the following proposition, the clearing price profile, denoted by $\lambda^{*}$, can achieve the maximal personal profit of every aggregator under the prosumption balance constraint.

\begin{proposition}\label{proppc}
Consider the Lagrangian dual problem \eqref{lagprob}.
With the profit function $J_{\alpha}$ defined as in \eqref{profJ}, let 
\begin{equation}\label{bidfun}
\hspace{6pt}
\boldsymbol{x}_{\alpha}(\lambda)\!:=\!
\bigl\{
x_{\alpha}\!\in\!\mathcal{X}_{\alpha}\!:\!J_{\alpha}(x_{\alpha};\lambda)\!\geq\!J_{\alpha}(x'_{\alpha};\lambda),\ \!\forall x'_{\alpha}\!\in\!\mathcal{X}_{\alpha}
\bigr\}\hspace{-20pt}
\end{equation}
be the set of all spatio-temporal prosumption profiles that attain the maximum profit under a spatio-temporal price profile $\lambda$.
Then, $(x_{\alpha}^{*})_{\alpha\in \mathcal{A}}$ and $\lambda^{*}$ are primal and dual solutions to \eqref{lagprob} if and only if 
\begin{equation}\label{condsol}
x_{\alpha}^{*}\in\boldsymbol{x}_{\alpha}(\lambda^{*}),\ \ \forall\alpha\in \mathcal{A},\displaystyle \quad{\rm and}\quad
\sum_{\alpha= 1}^{|\mathcal{A}|} 
x_{\alpha}^{*}=0.
\end{equation}
\end{proposition}\vspace{-6pt}

\begin{proof}
The convex conjugate of $F_{\alpha}$ is defined as
\[
\overline{F}_{\alpha}(\lambda):=\sup_{x_{\alpha} \in \mathcal{C}_{\alpha} }\left\{\langle\lambda,x_{\alpha}\rangle-F_{\alpha}(x)\right\},
\]
which is called the Legendre-Fenchel transformation \cite{rockafellar1970convex}.
For the maximization in \eqref{lagprob}, its first-order optimality condition yields the generalized equation
\begin{equation}\label{geneq}
0\displaystyle \in\sum_{\alpha\in \mathcal{A}}\partial\overline{F}_{\alpha}(\lambda^{*}).
\end{equation}
To prove the equivalence between \eqref{condsol} and \eqref{geneq}, it suffices to show the identity of 
\begin{equation}\label{xeqF}
\boldsymbol{x}_{\alpha}(\lambda)=\partial\overline{F}_{\alpha}(\lambda).
\end{equation}
Consider some $x_{\alpha}^{\star}\in\boldsymbol{x}_{\alpha}(\lambda)$, i.e., a maximizer of $J_{\alpha}$ under a given $\lambda$.
Then, its optimality condition yields $\lambda\in\partial F_{\alpha}(x_{\alpha}^{\star})$.
Because $\partial\overline{F}_{\alpha}$ is shown to be the inverse mapping of $\partial F_{\alpha}$; see \cite[Theorem~23.5]{rockafellar1970convex}, it is equivalent to $x_{\alpha}^{\star}\in\partial\overline{F}_{\alpha}(\lambda)$.
Hence, \eqref{xeqF} is proven.
\hfill $\blacksquare$
\end{proof}

\begin{remark}
We again remark the significance of the proposed market model.
One may think that an economic dispatch problem with uncertain renewable resources can be formulated as a  robust optimization problem, such as robust UC problems \cite{jiang2012robust,bertsimas2013adaptive,zheng2015stochastic}, where all available resources in the grid can be involved as the decision variables.
However, in such a centralized formulation, an owner of renewable resources can be a ``free rider,'' who does not have responsibility for uncertainty management. 
In particular, the uncertainty of renewable resources can be covered by dispatchable resources owned by some other aggregators.
In general, it is not easy to determine a reasonable price or side payment for the use of such resources of others,  because a unique price invariant for renewable scenarios is difficult to determine.
This actually hampers the spread of storage resources, such as battery systems, to absorb uncertainty of renewable energy.
In contrast, our market model imposes self-responsibility to uncertainty management on each aggregator.
This enables to create synergy between renewable and storage resources;
Each aggregator has an incentive to own a storage resource in conjunction with a renewable resource because reducing a waste of energy curtailment due to uncertainty makes his/her own profit higher.
\end{remark}


\begin{figure*}[t]
\centering
\includegraphics[width=170mm]{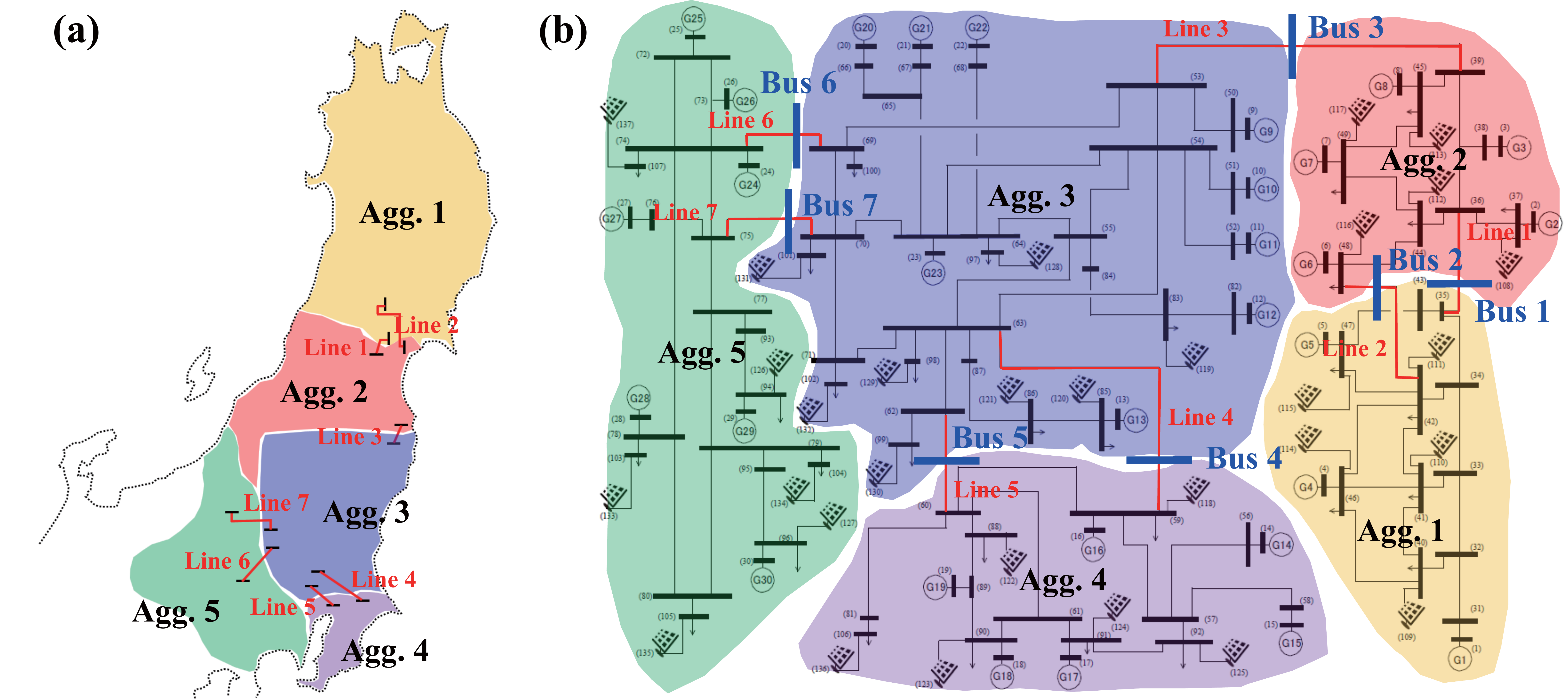}

\vspace{3pt}
\caption{
\textbf{(a)} North east area of Japan divided into five subareas.
\textbf{(b)} PV-integrated IEEJ EAST 30-machine model.
The subareas labeled as Agg.~$\alpha$ and the major transmission lines labeled as Line~$i$  are compatible in both subfigures.
}
\label{fignetworks}
\vspace{8pt}
\end{figure*}

\begin{table}[t]\centering
\caption{List of thermal generators.}\vspace{9pt}
\tiny
\begin{tabular}{|c || c|c|c|}\hline
Label: Type & Capacity [MW] & Cost [JPY/kWh] & Duration [h]  \\ \hline \hline
G1: Coal (C) & 3000 & 1.95 & 12 \\ \hline
G2: Coal (A) & 2500 & 4.23 & 12 \\ \hline
G3: Coal (B) & 2000 & 3.58 & 12 \\ \hline
G4: LNG (A) & 2250 & 5.48 & 1 \\ \hline
G5: Coal (D) & 1500 & 4.88 & 12 \\ \hline
G6: Oil (A) & 500 & 12.39 & 1 \\ \hline
G7: Oil (B) & 500 & 11.31 & 1 \\ \hline
G8: LNG (B) & 1000 & 5.40 & 1 \\ \hline
G9: Coal (A) & 2760 & 4.23 & 12 \\ \hline
G10: LNGCC (C) & 8312 & 1.91 & 6 \\ \hline
G11: LNGCC (B) & 4500 & 4.07 & 1 \\ \hline
G12: LNGCC (A) & 3700 & 2.48 & 1 \\ \hline
G13: LNG (C) & 2400 & 6.02 & 1 \\ \hline
G14: LNGCC (A) & 1265 & 2.48 & 1 \\ \hline
G15: LNGCC (C) & 1710 & 1.91 & 6 \\ \hline
G16: Coal (D) & 1160 & 4.88 & 12 \\ \hline
G17: LNG (B) & 820 & 5.40 & 1 \\ \hline
G18: Oil (B) & 831 & 11.31 & 1 \\ \hline
G19: Oil (A) & 957 & 12.39 & 1 \\ \hline
G20: LNG (B) & 1050 & 5.40 & 1 \\ \hline
G21: LNG (A) & 1000 & 5.48 & 1 \\ \hline
G22: Oil (C) & 900 & 11.48 & 1 \\ \hline
G23: Oil (B) & 675 & 11.31 & 1 \\ \hline
G24: Oil (A) & 751 & 12.39 & 1 \\ \hline
G25: Coal (C) & 4450 & 1.95 & 12 \\ \hline
G26: Oil (B) & 1000 & 11.31 & 1 \\ \hline
G27: LNG (C) & 1340 & 6.02 & 1 \\ \hline
G28: LNGCC (B) & 1560 & 4.07 & 1 \\ \hline
G29: Coal (B) & 3130 & 3.58 & 12 \\ \hline
G30: Oil (C) & 1250 & 11.48 & 1 \\ \hline
\end{tabular}
\label{tablist}\vspace{3pt}
\end{table}

\section{Numerical Analysis}\label{secnum}

\subsection{Simulation Setup}

\subsubsection{PV-Integrated IEEJ EAST 30-Machine Model }

In this section, we conduct a numerical analysis to investigate what level of battery penetration can be realized from the viewpoints of not only social economy but also personal economy, under the market regulation in which each aggregator is responsible for absorbing the uncertainty of his/her renewable power generation.
We suppose a spatio-temporal energy market in the north east area of Japan depicted as in Fig.~\ref{fignetworks}\textbf{(a)}, which is divided into five subareas labeled as Agg.~1--Agg.~5.
A bulk power system model of this area is found in the literature, called the IEEJ EAST 30-machine model \cite{IEEJEast30}, whose network structure is shown in Fig.~\ref{fignetworks}\textbf{(b)}.
We modify it as integrating photovoltaic (PV) generators at buses around residential areas.
In this numerical simulation, we suppose that each subarea is managed by a corresponding aggregator who participates in the energy market.
Furthermore, we suppose that, prosumption energy amounts are transacted at 24 time spots, and  the transactions of prosumption profiles are conducted at seven buses, denoted as Buses~1--7 in Fig.~\ref{fignetworks}\textbf{(b)}.
Therefore, each spatio-temporal prosumption profile $x_{\alpha}$ in \eqref{augpro} and the clearing price profile $\lambda$ in \eqref{auglam} are both $(24\times 7)$-dimensional.

Each aggregator is supposed to have several types of generators, loads, batteries, and PV generators, as being compatible with the mathematical modeling in Section~\ref{secpfm}.
In the following, we describe the details of Aggregator~1 because those of the other aggregators are similar.
The supplementary material of this paper lists the complete information of all aggregators, while explaining the practical meanings of parameter settings, which are based on real data sets.

\subsubsection{Setup of Generators}


The fuel cost and output capacity of each generator are listed in Table~\ref{tablist}.
In this table, ``Duration" corresponds to the period of output regulation.
For example, the output of Coal (C) is regulated every 12 hours, i.e., its output is fixed during the day and night time.

As seen from Fig.~\ref{fignetworks}\textbf{(b)}, Aggregator~1 has three generators, labeled as G1, G4, and G5.
Let $g_1^{(1)}$, $g_1^{(2)}$, and $g_1^{(3)}$ denote the power generation profiles of those generators.
Then, the generation cost function of Aggregator~1 is given as
\[
\spliteq{
G_{1}&(g_{1})=  
\displaystyle{\min_{
(g_1^{(1)}, g_1^{(2)}, g_1^{(3)}) 
}}  
\sum_{k=1 }^{3} c^{(k)}_1 {\mathds 1}^{\sf T}g^{(k)}_1\\
&{\rm s.t.}\quad g_{1}=\sum_{ k=1 }^{3} g^{ (k) }_1,\: \: 
(g_1^{(1)}, g_1^{(2)}, g_1^{(3)}) \in \mathcal{G}^{(1)}_1 \times \mathcal{G}^{(2)}_1 \times \mathcal{G}^{(3)}_1
}
\]
where the fuel cost and the output limit of the $k$th generator are denoted by $c^{(k)}_1$[JPY/kWh] and $\mathcal{G}^{(k)}_1$, respectively.

\subsubsection{Setup of Loads}

The load profile of each aggregator is given as a typical profile of consumers such as residences, commercial facilities, and factories.
The aggregate of all aggregators' load profiles, which amounts to 950.6[GWh], is plotted by the black dotted line of the upper envelope of Fig.~\ref{figover1}\textbf{(a1)}, which is identical to those in Figs.~\ref{figover1}\textbf{(a2)}--\textbf{(a4)}, and those in Figs.~\ref{figover2}--\ref{figover3r}\textbf{(a1)}--\textbf{(a4)}.
The load profile of Aggregator~1 is plotted by the black line with circles in Fig.~\ref{figagg1}\textbf{(a1)}, which is the same as those in Figs.~\ref{figagg1}\textbf{(a2)} and \textbf{(a3)}.
This load profile corresponds to the aggregate of 8.19 million residences in terms of consumption energy amounts.

\subsubsection{Setup of PV Generators}

We consider three penetration levels of PV generators, referred to as \textsf{PV(1)}, \textsf{PV(2)}, and \textsf{PV(3)}.
This PV penetration level is based on the premise of the feed-in tariff scheme in Japan.
In particular, according to the roadmap of Japan Photovoltaic Energy Association \cite{JPEA2017PV}, we set \textsf{PV(1)}--\textsf{PV(3)} as being comparable to the estimations in 2020, 2030, and 2050, respectively.
In this simulation, ten scenarios are considered for each aggregator, i.e., $m=10$ in \eqref{PVsc}.

At the level of \textsf{PV(1)}, the aggregate average of all PV scenarios subtracted from the aggregate of all load profiles, i.e., 
\[
\sum_{\alpha =1}^{5} \left\{l_{\alpha}-\frac{1}{10}\sum_{s=1}^{10}p_{\alpha}^{(s)}\right\},
\]
is plotted by the green dotted line in Fig.~\ref{figover1}\textbf{(a1)}, which is identical to those in Figs.~\ref{figover1}\textbf{(a2)}--\textbf{(a4)}.
In a similar way, the corresponding profiles are plotted in Figs.~\ref{figover2}\textbf{(a1)}--\textbf{(a4)} and Figs.~\ref{figover3}\textbf{(a1)}--\textbf{(a4)} at the levels of  \textsf{PV(2)} and \textsf{PV(3)}, respectively.
Furthermore, the PV scenarios of Aggregator~1 are plotted by the thin solid lines in Figs.~\ref{figagg1}\textbf{(a1)}--\textbf{(a3)}, which correspond to \textsf{PV(1)}--\textsf{PV(3)}, respectively.

\subsubsection{Setup of Batteries}

The basic objective of this numerical analysis is to investigate how the penetration levels of batteries affect market results.
The battery penetration level is defined as follows.
The aggregate of all loads, i.e., $\sum_{\alpha =1}^{5} l_{\alpha}$, can be regarded as 95 million residential loads in terms of consumption energy.
Based on this, we say that the battery penetration level is $r$\% if each of $r$\% of 95 million residences has a battery with $\pm$7[kW] inverter capacity and 14[kWh] battery capacity.
Its charge and discharge efficiencies are supposed to be $0.95$.

The battery cost function $D_{\alpha}$ is specifically given as follows.
Let $\sigma_{\alpha}^{\rm 0}$ denote the initial amount of stored energy, i.e., the state of charge (SOC), which is defined as the deviation from a neutral value.
It is reasonable to suppose that a higher level of the final SOC is more preferable than a lower level, and vice versa.
To take into account this aspect, we suppose that each aggregator assesses the value of the final SOC by 
\begin{equation}\label{spD}
D_{\alpha}(\delta_{\alpha})
=-d\left(\sigma_{\alpha}^{\rm fin}(\delta_{\alpha})\right)
\end{equation}
where $\sigma_{\alpha}^{\rm fin}$ denotes the final SOC defined by
 \[
 \sigma_{\alpha}^{\rm fin}(\delta_{\alpha})
:=
\sigma_{\alpha}^{\rm 0}+{\mathds 1}^{\sf T}(\delta^{\rm in}_{\alpha}-\delta^{\rm out}_{\alpha}),
 \]
and $d:\mathbb{R}\rightarrow \mathbb{R}$ is a concave function given as
\[
d(\sigma):=
\left\{
\begin{array}{cl}
a_{4}(\sigma-\overline{\sigma})+a_{3}\overline{\sigma},&\hspace{-3pt}\quad\overline{\sigma}\leq \sigma,\vspace{-3pt}\\
a_{3}\sigma,&\hspace{-3pt}\quad 0\leq \sigma<\overline{\sigma},\vspace{-3pt}\\
a_{2}\sigma,&\hspace{-3pt}\quad\underline{\sigma}\leq \sigma<0,\vspace{-3pt}\\
a_{1}(\sigma-\underline{\sigma})+a_{2}\underline{\sigma},&\hspace{-3pt}\quad \sigma<\underline{\sigma}.
\end{array}
\right.
\]
This concave function $d$, the graph of which is depicted in Fig.~\ref{figgraphd}, reflects a basic strategy of aggregators for battery management.
For example, if the temporal average of a clearing price is between $a_3$ and $a_2$, then each aggregator aims at restoring the final SOC to 50\%, or if it is between $a_2$ and $a_1$, then each aggregator aims at regulating the final SOC to $\underline{\sigma}$ as selling the battery stored energy.

\begin{figure}[t]
\centering
\includegraphics[width=45mm]{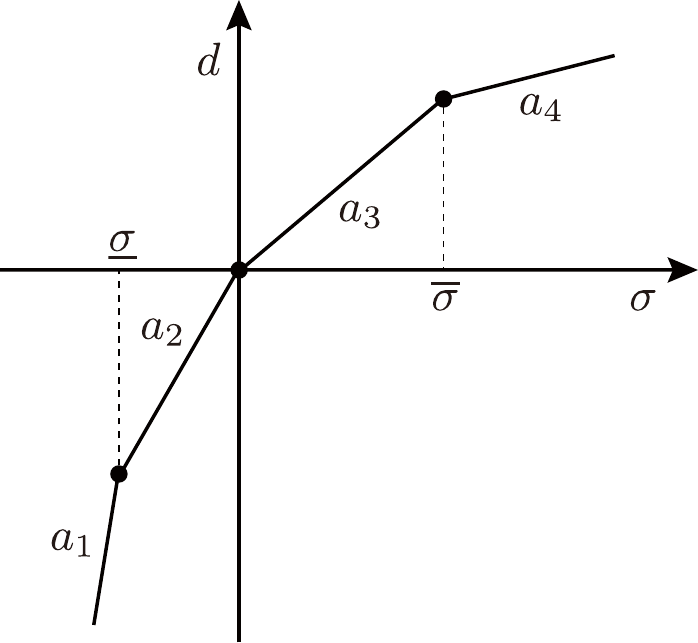}
\vspace{3pt}
\caption{Graph of $d$. The slope of each segment is denoted as $a_i$.}
\label{figgraphd}
\vspace{8pt}
\end{figure}

In this simulation, we set $\ a_{1}=11$,  $a_{2}=8$,  $a_{3}=4$, $a_{4}=1$, and $\sigma_{\alpha}^{\rm 0}=0$ for all $\alpha=1,\ldots,5$.
The values of $\underline{\sigma}$ and $\overline{\sigma}$ are set to represent 37.5\% and 62.5\% of the full SOC.


\begin{table}[t]\centering
\caption{List of transmission line capacities.}\vspace{9pt}
\tiny
\begin{tabular}{|c || c|c|c| c|c|c|c|}\hline
Line label &   1  & 2 & 3 & 4 & 5 & 6 & 7  \\ \hline \hline
Capacity [GW] &  0.63 & 0.38 & 1.16 & 0.49 & 0.81 & 0.88 & 1.05\\ \hline
\end{tabular}
\label{tabcap}\vspace{8pt}
\end{table}

\subsubsection{Setup of Transmission Line Constraints}

As spatial constraints, we take into account the capacities of the major transmission lines denoted as Lines~1--7 in Fig.~\ref{fignetworks}, while regarding each subarea as an aggregated single bus system, inside of which the capacities of transmission lines are large enough.
In this simulation, we set the capacities of Lines~1--7 as listed in Table~\ref{tabcap}.

\subsubsection{Setup of Prosumption Cost Functions}

The prosumption cost function of each of all aggregators is constructed based on the definition of $\tilde{F}_{\alpha}$ in \eqref{pcosfun2} for which the uniform probability distribution, i.e.,  
\[
\pi_{\alpha}^{(s)}=\frac{1}{10},\quad \forall s \in \{1,\ldots,10 \}
\] 
is supposed for every aggregator.

\begin{figure*}[t]
\centering
\includegraphics[width=180mm]{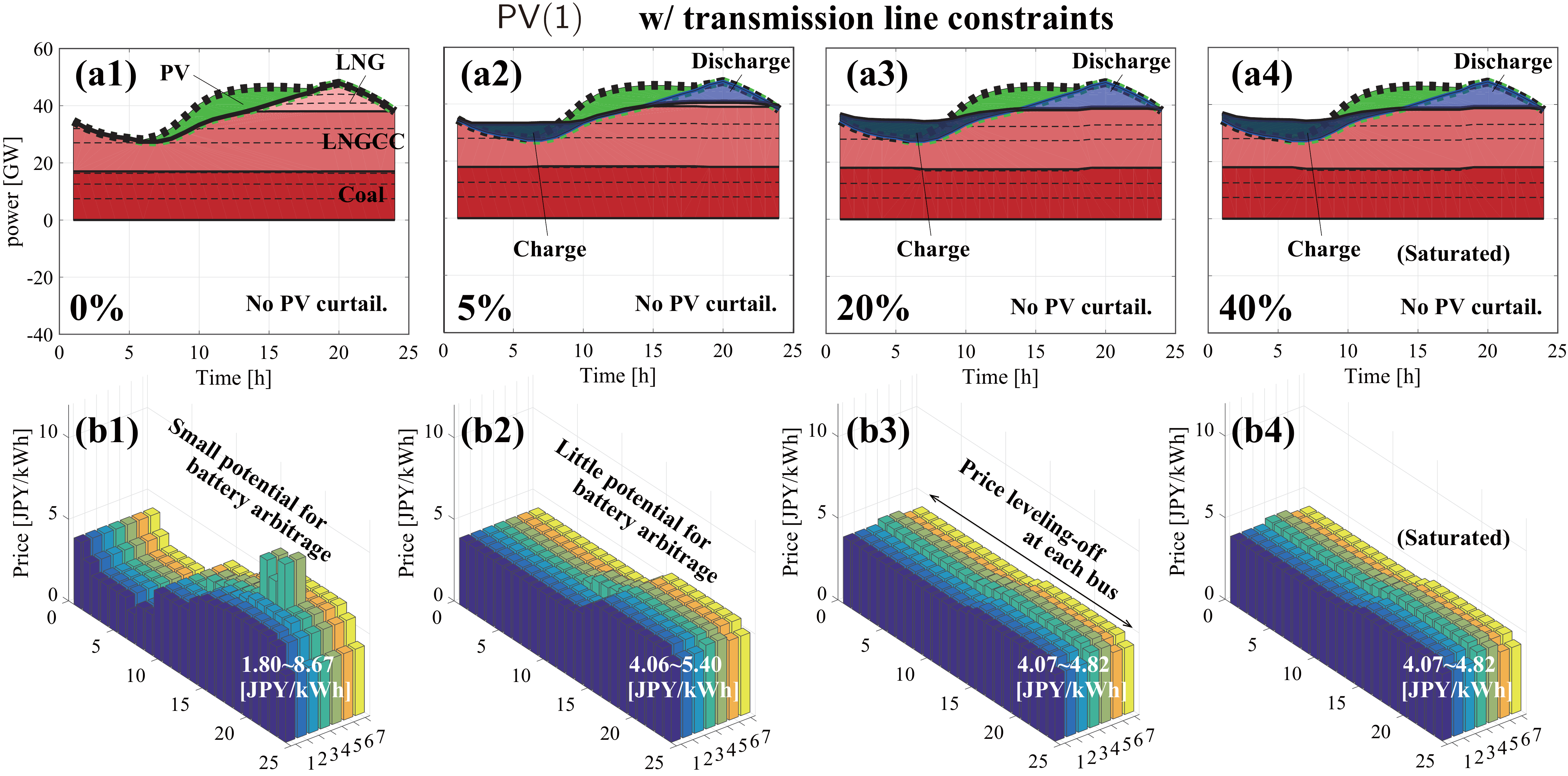}
\vspace{3pt}
\caption{
\textbf{(a1)--(a4)} Overviews of prosumption balance at the penetration level of \textsf{PV(1)}. 
\textbf{(b1)--(b4)} Clearing price profiles.
\textbf{(a1),(b1)} 0\% battery penetration level. 
\textbf{(a2),(b2)} 5\% battery penetration level.
\textbf{(a3),(b3)} 20\% battery penetration level.
\textbf{(a4),(b4)} 40\% battery penetration level.
}
\label{figover1}
\vspace{8pt}
\end{figure*}

\begin{figure*}[t]
\centering
\includegraphics[width=180mm]{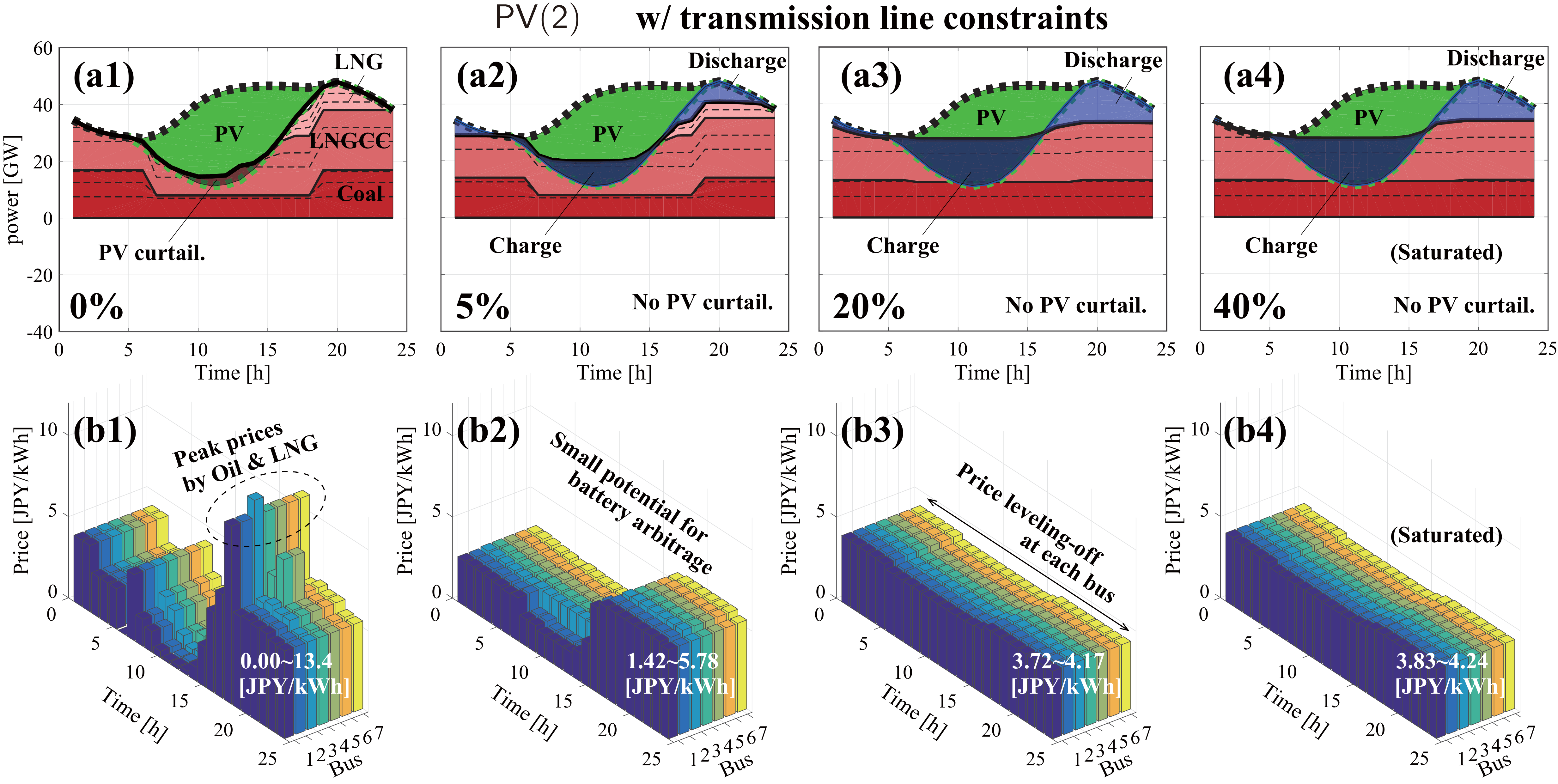}
\vspace{3pt}
\caption{
\textbf{(a1)--(a4)} Overviews of prosumption balance at the penetration level of \textsf{PV(2)}. 
\textbf{(b1)--(b4)} Clearing price profiles.
\textbf{(a1),(b1)} 0\% battery penetration level. 
\textbf{(a2),(b2)} 5\% battery penetration level.
\textbf{(a3),(b3)} 20\% battery penetration level.
\textbf{(a4),(b4)} 40\% battery penetration level.
}
\label{figover2}
\vspace{8pt}
\end{figure*}

\begin{figure*}[t]
\centering
\includegraphics[width=180mm]{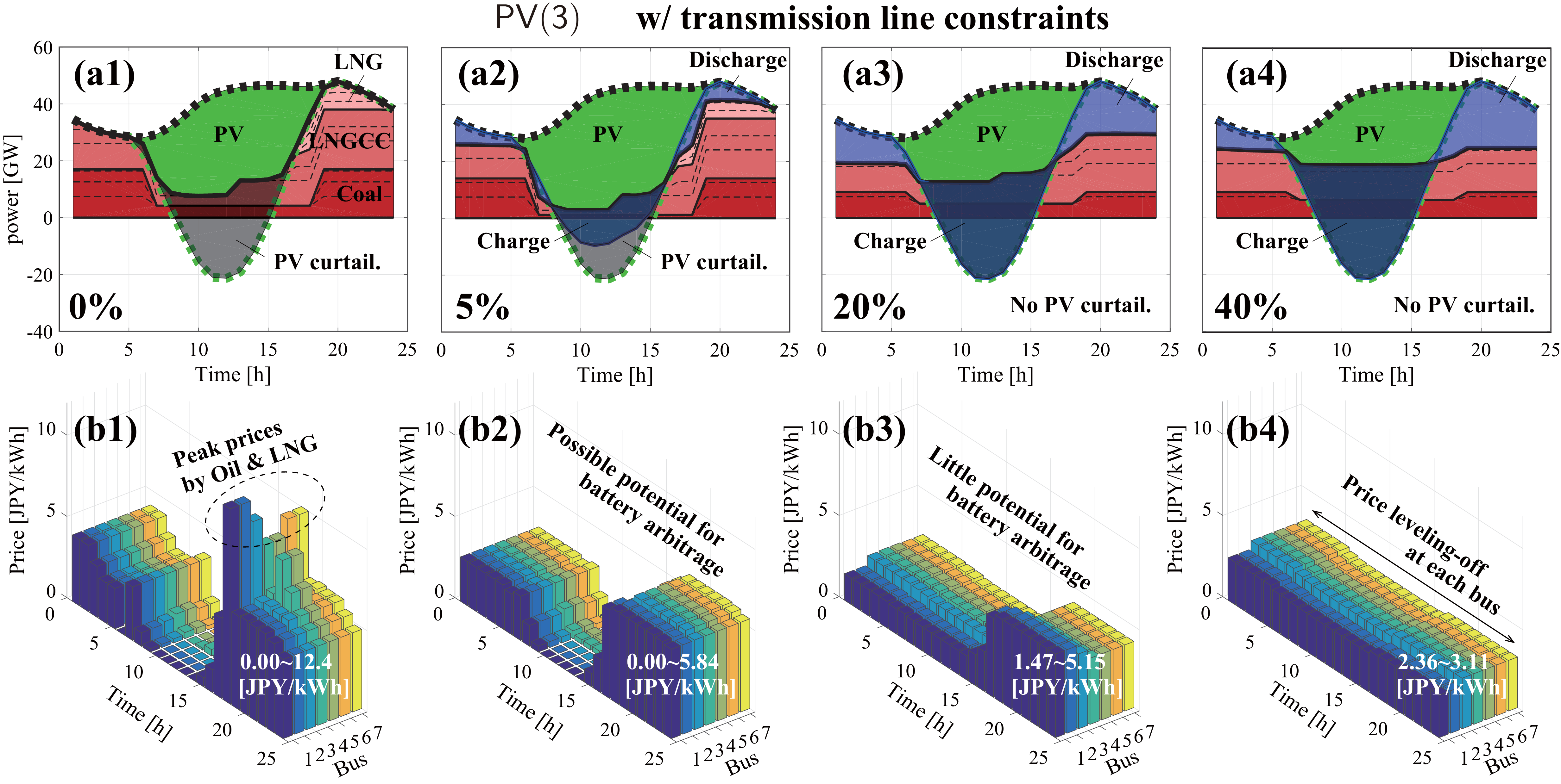}
\vspace{3pt}
\caption{
\textbf{(a1)--(a4)} Overviews of prosumption balance at the penetration level of \textsf{PV(3)}. 
\textbf{(b1)--(b4)} Clearing price profiles.
\textbf{(a1),(b1)} 0\% battery penetration level. 
\textbf{(a2),(b2)} 5\% battery penetration level.
\textbf{(a3),(b3)} 20\% battery penetration level.
\textbf{(a4),(b4)} 40\% battery penetration level.
}
\label{figover3}
\vspace{8pt}
\end{figure*}

\begin{figure*}[t]
\centering
\includegraphics[width=180mm]{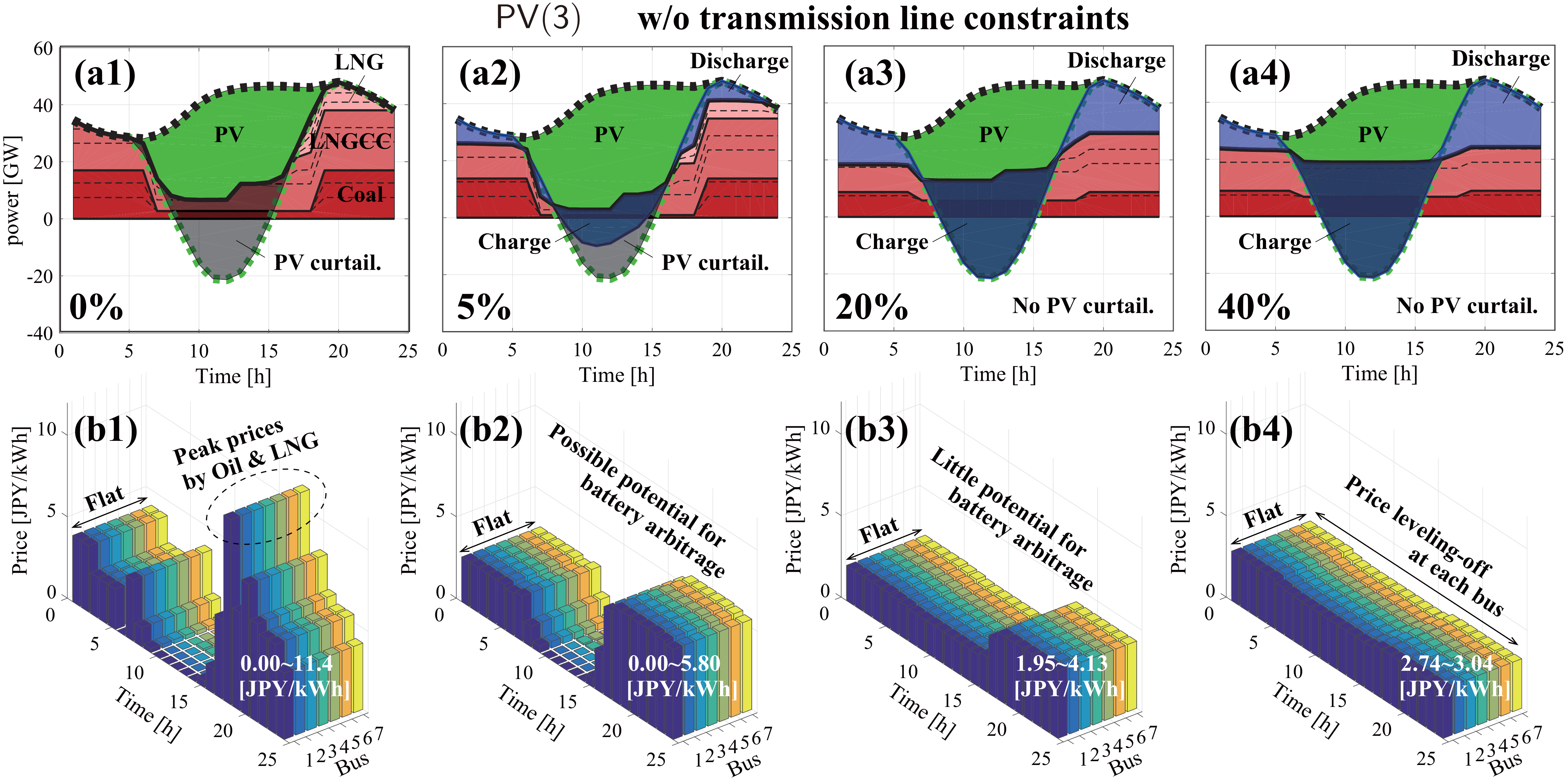}
\vspace{3pt}
\caption{
\textbf{(a1)--(a4)} Overviews of prosumption balance at the penetration level of \textsf{PV(3)} when transmission line constraints are negligible. 
\textbf{(b1)--(b4)} Clearing price profiles.
\textbf{(a1),(b1)} 0\% battery penetration level. 
\textbf{(a2),(b2)} 5\% battery penetration level.
\textbf{(a3),(b3)} 20\% battery penetration level.
\textbf{(a4),(b4)} 40\% battery penetration level.
}
\label{figover3r}
\vspace{8pt}
\end{figure*}

\begin{figure*}[t]
\centering
\includegraphics[width=180mm]{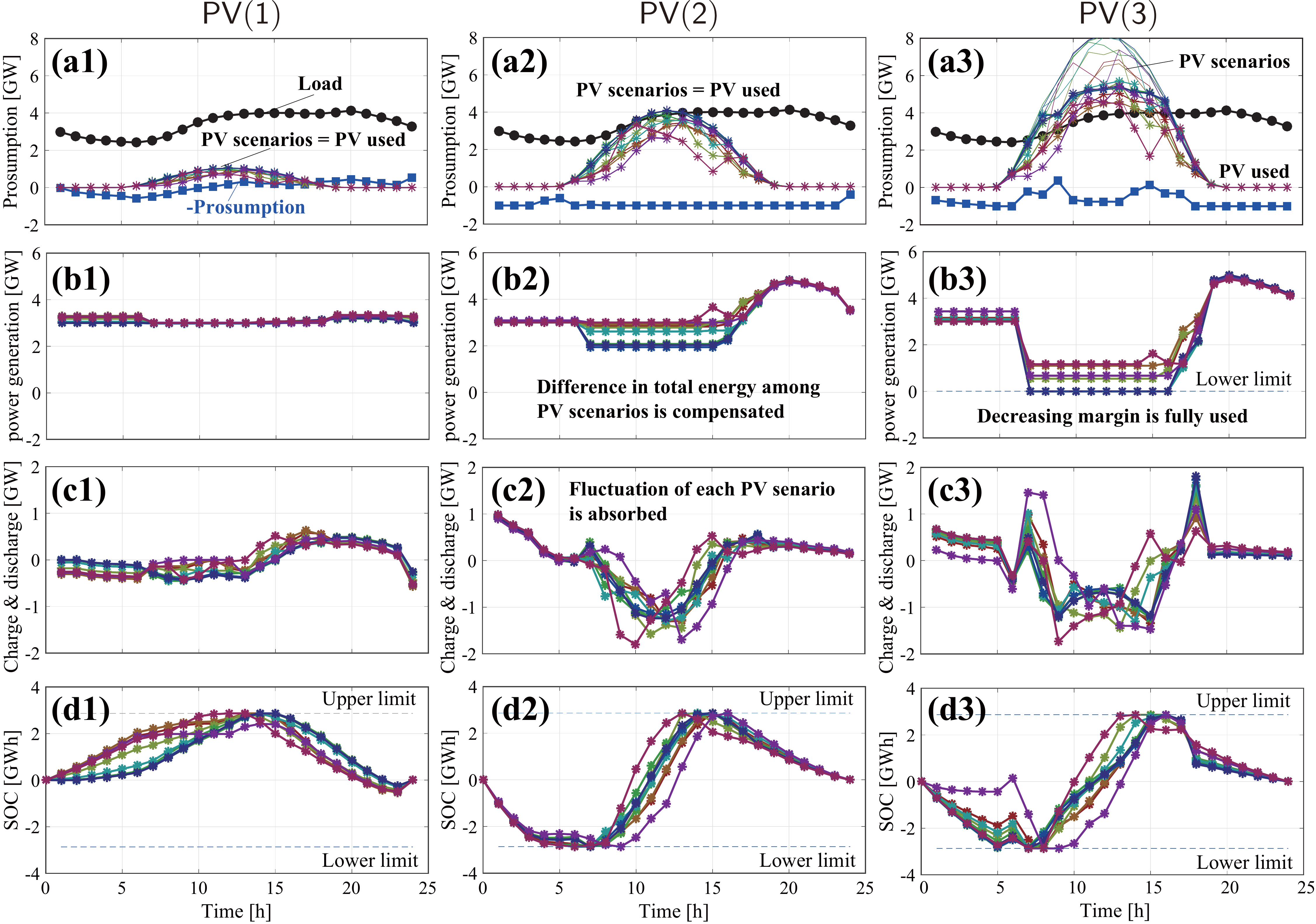}
\vspace{3pt}
\caption{
Profiles of Aggregator~1 at 5\% battery penetration level. 
\textbf{(a1)--(d1)} Penetration level of \textsf{PV(1)}. 
\textbf{(a2)--(d2)} Penetration level of \textsf{PV(2)}. 
\textbf{(a3)--(d3)} Penetration level of \textsf{PV(3)}. 
\textbf{(a1)--(a3)} Prosumption (sign opposite), load, PV scenarios, PV curtailments subtracted from PV scenarios. 
\textbf{(b1)--(b3)} Thermal power generation.
\textbf{(c1)--(c3)} Battery charge and discharge. 
\textbf{(d1)--(d3)} SOC.
}
\label{figagg1}
\vspace{8pt}
\end{figure*}

\begin{figure*}[t]
\centering
\includegraphics[width=180mm]{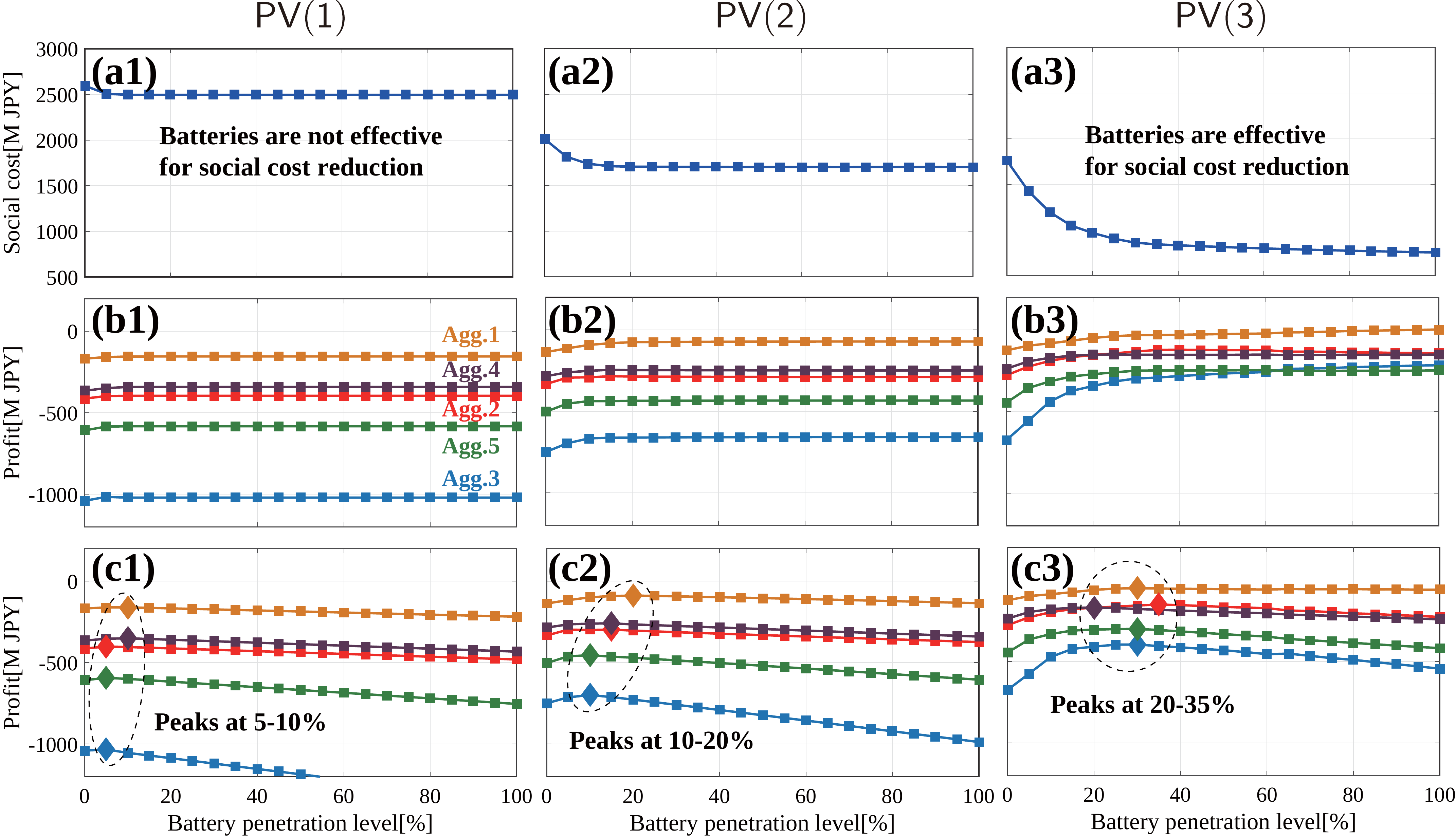}
\vspace{-0pt}
\caption{
\textbf{(a1)--(a3)} Social costs. 
\textbf{(b1)--(b3)} Profit of each aggregator. 
\textbf{(c1)--(c3)} Net profit of each aggregator considering battery purchase costs.
Each diamond mark corresponds to the battery penetration level at which the maximum personal profit with respect to battery system enhancement is attained.
}
\label{figcospro}
\vspace{3pt}
\end{figure*}

\subsection{Simulation Results}

\subsubsection{Overviews of Prosumption Balancing}

Solving the social cost minimization problem \eqref{marketp}, we obtain the transacted spatio-temporal prosumption profiles $x_1^*,\ldots,x_{5}^*$ and the clearing price $\lambda^*$.
For visualization of the market results, we introduce the notation of 
\[
\spliteq{
\bigl( g_{\alpha}^*(s),& \delta_{\alpha}^*(s)  ,q_{\alpha}^*(s) \bigr) \\
&:=\argmin_{(g_{\alpha},\delta_{\alpha},q_{\alpha}) \in \mathcal{F}_{\alpha}(x_{\alpha}^{*};p_{\alpha}^{(s)})}\Bigl\{G_{\alpha}(g_{\alpha})+D_{\alpha}(\delta_{\alpha})\Bigr\},
}
\]
which represents the optimal adjustable variable when the renewable scenario $p_{\alpha}^{(s)}$ arises.
The components of $x_{\alpha}^*$ and $\delta_{\alpha}^*$ are denoted in such a way that
\[
x_{\alpha}^* = ( x_{\alpha}^{(1)*}, \ldots, x_{\alpha}^{(7)*} ),\quad
\delta_{\alpha}^*(s) = \left( \delta_{\alpha}^{{\rm in}*}(s), \delta_{\alpha}^{{\rm out}*}(s) \right).
\]

At the levels of \textsf{PV(1)}--\textsf{PV(3)}, we plot the overviews of prosumption balancing in Figs.~\ref{figover1}--\ref{figover3}\textbf{(a1)}--\textbf{(a4)}, the subfigures of which correspond to 0\%, 5\%, 20\%, and 40\% battery penetration levels, respectively.
The black thick line in each subfigure represents the average of the total energy amounts generated by the thermal generators, i.e.,
\[
\sum_{\alpha =1}^{5}
\left\{
\frac{1}{10}
\sum_{s=1}^{10} 
g^*_{\alpha}(s)
\right\}.
\]
The gray shadowed area corresponds to the average of the total PV curtailment amount, i.e., 
\[
\sum_{\alpha =1}^{5}
\left\{
\frac{1}{10}
\sum_{s=1}^{10} 
q^*_{\alpha}(s)
\right\}.
%
\]
The blue lines in Figs.~\ref{figover1}--\ref{figover3}\textbf{(a2)}--\textbf{(a4)} represent the average of the total battery charge and discharge profiles, i.e.,
\[
\sum_{\alpha =1}^{5}
\left\{
\frac{1}{10}
\sum_{s=1}^{10} 
\bigl(
\delta^{{\rm out}*}_{1}(s)-\delta^{{\rm in}*}_{1}(s)
\bigr)
\right\}.
%
\]

From these figures, we see that the total thermal power generation profile becomes flat as the battery penetration level increases.
Furthermore, from Fig.~\ref{figover1}\textbf{(a1)} and Fig.~\ref{figover2}\textbf{(a1)}, we see that the amounts of PV curtailment are considerably small even without batteries.
This is because the PV uncertainty is relatively small as thermal generators can manage it.
On the other hand, the PV uncertainty in Fig.~\ref{figover3}\textbf{(a1)} is too large to manage.
Then, Figs.~\ref{figover3}\textbf{(a2)}--\textbf{(a4)} show that the PV curtailment amount properly decreases as the battery penetration level increases.
These results imply that PV power generation with an appropriate amount of batteries has priority of energy supply higher than thermal generators while that without batteries does not always have such priority due to its uncertainty.
This result gives an answer to \textbf{Q2} in Section~\ref{seciq} with regard to merit order of different energy resources.

For reference, at the level of \textsf{PV(3)}, we show in Figs.~\ref{figover3r}\textbf{(a1)}--\textbf{(a4)} the overviews of prosumption balancing  when transmission line constraints on Lines~1--7 are removed.
In fact, Figs.~\ref{figover3r}\textbf{(a1)}--\textbf{(a4)} are almost the same as Figs.~\ref{figover3}\textbf{(a1)}--\textbf{(a4)}, meaning that the transmission line capacities in Table~\ref{tabcap} are large enough, at least, to obtain the optimal prosumption balancing.

\subsubsection{Spetio-Temporal Clearing Price Profiles}

At the levels of \textsf{PV(1)}--\textsf{PV(3)}, we plot in Figs.~\ref{figover1}--\ref{figover3}\textbf{(b1)}--\textbf{(b4)} the resultant clearing price profiles, which are spatio-temporally distributed.
In each subfigure, the minimum and maximum prices are written in white.
First, from Figs.~\ref{figover1}--\ref{figover3}\textbf{(b1)}, we see that peak prices around evening appear due to the use of expensive generators.
Then, from Figs.~\ref{figover1}--\ref{figover3}\textbf{(b2)}, we see that those price peaks disappear as the battery penetration level slightly increases.
Furthermore, we can also see from Figs.~\ref{figover1}--\ref{figover3}\textbf{(b2)}--\textbf{(b4)} that price differences gradually reduce as the battery penetration level increases, and finally an increased  battery penetration level results in price leveling-off, i.e., the clearing price profile at each of all buses has an almost flat distribution.
This is the result of arbitrage, i.e., temporal energy shift, enabled by batteries.

For reference, at the level of \textsf{PV(3)}, we show in Figs.~\ref{figover3r}\textbf{(b1)}--\textbf{(b4)} the resultant clearing price profiles when  transmission line constraints are removed.
In this case, the spatial distributions of clearing price profiles become flat.
This is because Buses~1--7 can be equivalently aggregated into a single bus; see Figs.~\ref{figthreeagg}\textbf{(a)} and \textbf{(b)} for the depiction of a multiple-bus system and its aggregated single-bus alternative.

\subsubsection{Resultant Profiles of Aggregator 1}

At 5\% battery penetration level, we plot the resultant profiles of Aggregator~1 in Fig.~\ref{figagg1}.
In particular,  Figs.~\ref{figagg1}\textbf{(a1)}--\textbf{(d1)} correspond to \textsf{PV(1)}, Figs.~\ref{figagg1}\textbf{(a2)}--\textbf{(d2)} correspond to \textsf{PV(2)}, and Figs.~\ref{figagg1}\textbf{(a3)}--\textbf{(d3)} correspond to \textsf{PV(3)}.
For clarity, the legends of Fig.~\ref{figagg1} are summarized in Table~\ref{tableg}.
From these figures, we see that the uncertainty of PV scenarios is absorbed by appropriate combination of generator management, battery management, and PV curtailment.
In particular, we see from Figs.~\ref{figagg1}\textbf{(b2)}--\textbf{(c2)} that thermal generators mainly serves for compensating the difference in the total energy amount of PV scenarios while batteries serves for absorbing the fluctuation of each PV scenario.
Only in Fig.~\ref{figagg1}\textbf{(a3)}, the PV curtailment is carried out as a last resort because the PV uncertainty is too large to be managed by generators and batteries.
Note that the final SOC in every case is mostly restored to the neutral as a result of the final SOC evaluation by $D_{\alpha}$ in \eqref{spD}.

\begin{table*}[t]\centering
\caption{Legend of Fig.~\ref{figagg1} 
(For representation of SOC, ``$M$" is used to denote lower triangular matrix whose nonzero elements are all one).}\vspace{9pt}
\tiny
\begin{tabular}{|c | m{30em} | c | c| }\hline
Mathematical representation & \centering Meaning  & Subfigures  & Line marker \\ \hline \hline
$l_1$  & 
Load  & 
\textbf{(a1)}--\textbf{(a3)} &
$\bullet$ \\ \hline
$-(x_1^{(1)*} + \cdots + x_1^{(7)*})$  & 
Sum of prosumption profiles at all buses (with opposite signs)  & 
\textbf{(a1)}--\textbf{(a3)} &
$\blacksquare$ \\ \hline
$g_1^*(s) \quad (s =1,\ldots,10)$ & 
Generation profiles (10 scenarios) &
\textbf{(b1)}--\textbf{(b3)} &
$\ast$ \\ \hline
$p_1^{(s)} \quad (s =1,\ldots,10)$ & 
PV senarios (10 scenarios) &
\textbf{(a1)}--\textbf{(a3)} &
none \\ \hline
$p_1^{(s)}-q_1^*(s) \quad (s =1,\ldots,10)$ & 
PV curtailment profiles subtracted from PV scenarios (10 scenarios) &
\textbf{(a1)}--\textbf{(a3)} &
$\ast$ \\ \hline
$\delta^{{\rm out}*}_{1}(s)-\delta^{{\rm in}*}_{1}(s) \quad (s =1,\ldots,10)$ & 
Battery charge and discharge profiles (10 scenarios) &
\textbf{(c1)}--\textbf{(c3)} &
$\ast$ \\ \hline
$M(\delta^{{\rm in}*}_{1}(s)-\delta^{{\rm out}*}_{1}(s)) \quad (s =1,\ldots,10)$ & 
SOC profiles (10 scenarios) &
\textbf{(d1)}--\textbf{(d3)} &
$\ast$ \\ \hline
\end{tabular}
\label{tableg}\vspace{3pt}
\end{table*}


\subsubsection{Social Costs versus Battery Penetration Levels}

Varying battery penetration levels, we plot the resultant social costs in Fig.~\ref{figcospro}\textbf{(a1)}--\textbf{(a3)}, which correspond to \textsf{PV(1)}--\textsf{PV(3)}, respectively.
From these figures, we see that the social cost decreases as the battery penetration level increases.
This is a natural consequence because the realizable spatio-temporal prosumption profile set $\mathcal{X}_{\alpha}$ in \eqref{realsp}, the direct product of which corresponds to the feasible domain of the social cost minimization problem \eqref{marketp}, enlarges as the battery penetration level increases.
We also see that the higher the PV penetration level is, the more effective the battery enhancement is.
This is because the amount of waste PV curtailment can be significantly reduced by energy storage systems.

\subsubsection{Personal Profits versus Battery Penetration Levels}

From the viewpoint of personal profit, we discuss a social equilibrium of battery penetration levels, giving an answer to \textbf{Q3}.
We plot the profits of each aggregator, i.e., $J_{\alpha}(x_{\alpha}^{*};\lambda^{*})$, in Figs.~\ref{figcospro}\textbf{(b1)}--\textbf{(b3)}, which correspond to \textsf{PV(1)}--\textsf{PV(3)}, respectively.
From these figures, at the levels of \textsf{PV(2)} and \textsf{PV(3)}, we see that the increasing rate of the personal profit is relatively high at lower battery penetration levels, while it decreases at higher battery penetration levels.
This profit saturation can be explained by the fact that the increment in profits by battery arbitrage as well as by PV curtailment reduction decrease as the battery penetration level increases.
This trend can be confirmed from the Figs.~\ref{figover1}--\ref{figover3}\textbf{(a1)}--\textbf{(a4)} and Figs.~\ref{figover1}--\ref{figover3}\textbf{(b1)}--\textbf{(b4)}, showing that both PV curtailment amount and temporal price volatility reduce as the battery penetration level increases.

Note that the enhancement of battery systems generally requires an additional cost for battery purchase.
To take into account this fact, we calculate the actual net profit as in Figs.~\ref{figcospro}\textbf{(c1)}--\textbf{(c3)}.
The battery price is supposed to be 5,000 [JPY/kWh] according to the price estimation in 2030 published as the roadmap \cite{NEDObat2013} from New Energy and Industrial Technology Development Organization (NEDO).
From these figures, we see that there is an equilibrium of the battery penetration level, highlighted by the diamond marks, that attains the maximum of the net profit for each PV penetration level.
It should be emphasized that this equilibrium can be realized as the assembly of long-term rational decisions of competitive market players who pursue just personal profit maximization, not social profit maximization.
This spread of battery systems is a consequence of mechanism design in this paper.

\section{Concluding Remarks}
In this paper, we presented modeling and analysis of day-ahead spatio-temporal energy markets.
We modeled the spatio-temporal energy market as an adjustable robust convex program, in which each competitive player is responsible for absorbing the uncertainty of renewable power generation.
Furthermore, by numerical analysis with a PV-integrated IEEJ EAST 30-machine model, we discussed the merit order of different energy resources and the existence of a social equilibrium of battery penetration levels, at which the maximum personal profit with respect to battery system enhancement is attained.
Consideration of non-convexity in unit commitment problems would be an interesting direction of future research.

\section*{Acknowledgements}
The authors would like to thank Dr.~Hideharu Sugihara, Osaka University, for technical assistance to construct the PV-integrated IEEJ EAST 30-machine model. 
This work was supported by JST CREST Grant Number JP-MJCR15K1, Japan.


\bibliographystyle{ieeetr}         
\bibliography{IEEEabrv,reference,reference_CREST,reference_aux}



\end{document}